\documentclass[11pt,leqno,a4paper,oneside]{amsart}

\usepackage{amssymb,amsthm,enumerate, wasysym}
\usepackage{array,arydshln}
\usepackage{pstricks, pst-plot, pst-node}
\usepackage{hyperref}                         
\usepackage[all]{xy}

\usepackage{tikz}                                 
\usetikzlibrary{angles}
\usetikzlibrary{decorations.markings,calc}
\usetikzlibrary{cd}

\input xy
\xyoption{all}

\entrymodifiers={!!<0pt,0.7ex>+}  

\hypersetup{pdftex,                           
bookmarks=true,
pdffitwindow=true,
colorlinks=true,
citecolor=black,
filecolor=black,
linkcolor=black,
urlcolor=blue,
hypertexnames=true}

\newcommand{\IN}{{\mathbb{N}}}
\newcommand{\IZ}{{\mathbb{Z}}}

\newcommand{\fp}{{\mathfrak{p}}}     

\newcommand{\cO}{{\mathcal{O}}}

\newcommand{\bA}{{\mathbf{A}}}
\newcommand{\bB}{{\mathbf{B}}}
\newcommand{\bb}{{\mathbf{b}}}
\newcommand{\bc}{{\mathbf{c}}}

\newcommand{\dade}{{\mathbf{D}_k}}

\DeclareMathOperator{\End}{End}               
\DeclareMathOperator{\Hom}{Hom}             
\DeclareMathOperator{\Res}{Res}               
\DeclareMathOperator{\Ind}{Ind}                  
\DeclareMathOperator{\Inf}{Inf}                    
\DeclareMathOperator{\Iso}{Iso}    		

\DeclareMathOperator{\Irr}{Irr}

\DeclareMathOperator{\soc}{soc}
\DeclareMathOperator{\head}{head}
\DeclareMathOperator{\rad}{rad}
\DeclareMathOperator{\Capp}{Cap}

\newcommand{\lconj}[2]{\,^{#1}\!#2}                     

\let\lra=\longrightarrow

\let\wh=\widehat

\newtheorem{thm}{Theorem}[section]
\newtheorem{lem}[thm]{Lemma}

\newtheorem{cor}[thm]{Corollary}
\newtheorem{prop}[thm]{Proposition}

\theoremstyle{theorem}
\newtheorem*{thm0}{Theorem}

\theoremstyle{theorem}

\theoremstyle{definition}

\theoremstyle{remark}
\newtheorem{rem}[thm]{Remark}

\begin{document}


\title[The classification of the trivial source modules in blocks with cyclic defect groups]{The classification of the trivial source modules in blocks with cyclic defect groups}

\date{\today}

\author{Gerhard Hiss}
\author{Caroline Lassueur}
\thanks{Corresponding author: Caroline Lassueur}

\address{{\sc Gerhard Hiss},  Lehrstuhl D f\"ur Mathematik, RWTH Aachen, 52056 Aachen, Germany.}
\email{gerhard.hiss@math.rwth-aachen.de}
\address{{\sc Caroline Lassueur}, FB Mathematik, TU Kaiserslautern, Postfach 3049, 67653 Kaiserslautern, Germany.}
\email[corresponding author]{lassueur@mathematik.uni-kl.de}



\begin{abstract} 
Relying on the classification of the indecomposable liftable modules in arbitrary blocks with non-trivial cyclic defect groups  by the first author and Naehrig we give a complete classification of the trivial source modules lying in such  blocks,  describing in particular  their associated path on the Brauer tree of the block in the sense of Janusz' classification of the indecomposable modules of such blocks.  Furthermore, the appendix contains a minor correction to the statement of the classification theorem of the indecomposable  liftable modules, as well as a description of  the minimal distance  from an arbitrary indecomposable liftable  module  to the boundary of the stable Auslander-Reiten quiver of the block. 
\end{abstract}

\thanks{Both authors gratefully acknowledge financial support by DFG SFB-TRR 195. This article is part of Project~A18 thereof.}

\maketitle


\pagestyle{myheadings}
\markboth{
The classification of the trivial source modules in blocks with cyclic defect groups}{The classification of the trivial source modules in blocks with cyclic defect groups}

\vspace{6mm}
\section{Introduction}
We consider a finite group $G$ and  an algebraically closed field $k$ of characteristic $p>0$.  The purpose of the present article is to provide an explicit classification of the indecomposable trivial source modules lying in a $p$-block $\bB$  of $kG$ with a non-trivial cyclic defect group $D\cong C_{p^n}$ ($n\geq 1$)  using Janusz' parametrisation  \cite[\S5]{Jan69} of the non-projective, non-simple indecomposable {$\bB$-modules} through an associated path on the Brauer tree $\sigma(\bB)$ of~$\bB$ together with a direction and a multiplicity as described in \cite{BleChi},  and which we briefly recall in Appendix~\ref{AppA}.  
\par
 The theory of blocks with cyclic defect groups originated with the work of Brauer \cite{Bra41,Bra42} and was later extended by a series of articles by Dade \cite{Dad66}, Junusz \cite{Jan69}, Kupisch \cite{Kup69}, Peacock \cite{Pea75,Pea77}, Feit \cite{Fei84}, Green \cite{GreenWalk}, \ldots, which  encode the structure of the Morita equivalence class of such blocks through  their Brauer trees. 
\par
As it turns out trivial source modules are not preserved by Morita equivalences in general, but they are preserved by the stronger source algebra equivalences (also called Puig equivalences). Thus, in order to achieve our aim, we use much more recent techniques. The first main ingredient  is the classification of cyclic blocks up to source algebra equivalence by Linckelmann \cite{LinckThese, LinckBourbaki, Linck96}, and in particular the description of the interior $D$-algebra structure of their source algebras. The second main ingredient is a description by Bleher-Chinburg \cite{BleChi} of the location  in the stable Auslander-Reiten quiver $\Gamma_s(\bB)$ of the block $\bB$ of an arbitrary non-projective indecomposable $\bB$-module given by its path. The third main ingredient is the classification of the indecomposable liftable $\bB$-modules by the first author and Naehrig \cite{HN12}.
\par
 In order to introduce our main results, let us fix some notation. For each $0\leq i\leq n$, let $D_i$ denote the unique cyclic subgroup of~$D$ of order~$p^i$  and let $\bb$ be the Brauer correspondent of~$\bB$ in $N_G(D_1)$. Then, by Linckelmann's results  \cite[Theorem~2.7]{Linck96},  any source algebra of $\bB$ (they are in fact all isomorphic as interior $D$-algebras) is determined, as interior $D$-algebra, by two parameters: 
\begin{itemize}
  \item[1.] a strengthened version of its Brauer tree, which we will see as $\sigma(\bB)$ together with a sign function, assigning a plus or a minus sign to each vertex of $\sigma(\bB)$ in an alternate way;  
  \item[2.] an  indecomposable  endo-permutation $kD$-module~$W$ isomorphic to a $kD$-source of the simple $\bb$-modules, and hence on which $D_1$ acts trivially. 
\end{itemize}
By \cite{Dad66} all indecomposable capped endo-permutation $kD$-modules on which~$D_1$ acts trivially arise as a source of the simple $\bb$-modules for some cyclic block $\bB$, and it follows from the classification of the endo-permutation modules, that there are, up to isomorphism, $2^{n-1}$ (resp.~$2^{n-2}$ if $p=2$) possibilities for $W$.
We also recall that  $\Gamma_s(\bB)$ is a graph whose set of vertices is  the set of isomorphism classes of non-projective indecomposable $\bB$-modules and  is a tube of shape $(\IZ/e\IZ)A_{|D|-1}$, where~$e$ is the inertial index of $\bB$. (See \cite[\S 4.13-4.17 and \S6.5]{BensonBookI}). 
\par
Our first main result, Theorem~\ref{thm:mainA}, describes the location in $\Gamma_s(\bB)$ of the non-projective indecomposable trivial source $\bB$-modules with vertex $D_i$  for each $1\leq i\leq n$ in terms of their distance to one of the boundaries, this in function of the strengthened Brauer tree of $\bB$, and the endo-permutation $kD$-module $W$ described above.  
This module $W$ is parametrised by an $n$-tuple $(a_0, \ldots , a_{n-1})$
of integers $a_i \in \{ 0, 1 \}$, $1 \leq i \leq n - 1$, with $a_0 = 0$. For 
each $1 \leq i \leq n$, let $\ell_i$ denote the dimension of the unique 
indecomposable direct summand of $\Res_{D_i}^D( W )$ with vertex~$D_i$. One 
of the boundaries of $\Gamma_s( \mathbf{B} )$ consists of indecomposable 
modules which lift to characteristic~$0$ in such a way, that the characters 
of these lifts take positive integer values on the generators of~$D_1$. 
If~$X$ is an indecomposable $\mathbf{B}$-module, we write $d^+( X )$ for 
the distance of~$X$ to this boundary in $\Gamma_s( \mathbf{B} )$. (We will 
give more precise definitions for these concepts and references for the
claims made above later on in our paper.) With these concepts we can  state our 
first main result as follows.

\begin{thm0}[Theorem~\ref{thm:mainA}] With the above notation, the following holds.
Let $i_0 < i_1 < \cdots < i_s$ be the indices such that $a_{i_0}=\ldots=a_{i_s}=1$ and $a_i=0$ if ${i\in\{1,\ldots,n-1\}\setminus\{i_0,\ldots, i_s\}}$. 
Let $1\leq i\leq n$ and let $X$ be  a non-projective indecomposable trivial source $\bB$-module with vertex~$D_i$.  Then
$$d^+(X)=\ell_i\cdot p^{n-i}-1\,,$$
and  $\ell_i=\sum_{0\leq i_j<i}(-1)^j p^{i-i_j}+(-1)^{|\{j\mid 0\leq i_j<i\}|}$\,.
\end{thm0}

Our second main result, Theorem~\ref{thm:mainB}, provides us with a classification of the indecomposable trivial source $\bB$-modules with vertex $D_i$ for each $1\leq i\leq n$ in terms of  their path on $\sigma(\bB)$, direction, and multiplicity.  Indecomposable trivial source modules with vertex $D_0=\{1\}$ are simply the projective indecomposable modules of $\bB$ and their structure is well-known. (See~\S\ref{ssec:PIMs}.)
We postpone the precise statement of this classification to Section~5 as it is rather long and complex.
\par
We note that several authors had already obtained some partial results in the direction of our two main results.  In \cite{Mic75} Michler characterises trivial source modules in blocks with cyclic defect groups through ring-theoretic properties. In \cite{BessenrodtCyclic} Bessenrodt describes the locations of the set of indecomposable modules with a given vertex $1\lneq D_i\leq D$ in $\Gamma_s(\bB)$, but does not deal with the sources of the modules. Finally, in \cite{KK10} Koshitani and Kunugi deal with properties of the characters of trivial source modules in the case in which $W=k$ is the trivial $kD$-module.   
However, none of these articles achieves a full classification of the indecomposable trivial source $\bB$-modules. 
\par
The paper is built up as follows. In Section~\ref{sec:prelim} we introduce our general notation and preliminary results on trivial source modules and endo-permutation modules.  In Section~\ref{sec:cyclicblocks} we recapitulate known results  on cyclic blocks, which we will use throughout. In Section~\ref{sec:location}, we provide the main steps of the proof of our first main result on the location of the trivial source $\bB$-modules in $\Gamma_s(\bB)$ through a reduction to a source algebra of a block of the centraliser $C_G(D_1)$ covered by the block $\bb$. In Section~\ref{sec:mainresults} we state and prove our two main results. Finally, Appendix~A contains a minor correction to the statement of the main theorem of~\cite{HN12}, which applies to the situation that the exceptional vertex is a leaf of the Brauer tree, whereas Appendix~B contains a full description of the distances between the indecomposable liftable $\bB$-module and  the boundary $\Gamma_s(\bB)$.\\


\vspace{6mm}
\section{Preliminaries}\label{sec:prelim}

\subsection{General notation}

Throughout we let $p$ be a prime number and  $G$ be a finite group of order divisible by~$p$. 
We let $(K,\cO,k)$ be a $p$-modular system, where  $\cO$ denotes a complete discrete valuation ring of characteristic zero with unique maximal ideal $\frak{p}:=J(\cO)$, algebraically closed residue field $k:=\cO/\fp$ of characteristic $p$, and field of fractions $K=\text{Frac}(\cO)$, which we assume to be large enough for $G$ and its subgroups in the sense that $K$ contains a root of unity of order $\exp(G)$, the exponent of~$G$. 
Unless otherwise stated, for $R\in\{\cO,k\}$, $RG$-modules are assumed to be finitely generated left $RG$-lattices, that is free as $R$-modules, and by a block~$\bB$ of~$G$, we mean a block of~$kG$. 
Given a subgroup $H\leq G$, we let $R$  denote the trivial $RG$-lattice, we write $\Res^G_H(M)$ for the restriction of the $RG$-lattice $M$ to $H$, and  $\Ind_H^G(N)$ for the induction of the $RH$-lattice $N$ to $G$. Given a normal subgroup $U$ of $G$, we write $\Inf_{G/U}^{G}(M)$ for the inflation of the $R[G/U]$-module $M$ to $G$.  We write $M^{\ast}:=\Hom_k(M,k)$ for the $k$-dual of a $kG$-module $M$,  $\soc(M)$ for its socle,  $\head(M)$ for  its head, and  $\rad(M)$ for its radical. 
If $M$ is a uniserial $kG$-module, then we denote by $\ell(M)$ its  composition length. 
We let $\Omega$ denote the  usual Heller operator. We denote by $\Irr(G)$ (resp. $\Irr(\bB)$)  the set of irreducible $K$-characters  of $G$ (resp. of the block $\bB$ of $kG$). 
We write $\bB_0(kG)$ for the principal block of $kG$. 
We recall that the reduction modulo $\fp$ of an $\cO G$-lattice $L$  is $L/\mathfrak{p}L\cong k\otimes_{\cO}L$, and a $kG$-module $M$ is said to be liftable if there exists an $\cO G$-lattice $\widetilde{M}$ such that $M\cong \widetilde{M}/\mathfrak{p}\widetilde{M}$.
Finally, we assume that the reader is acquainted with the terminology of equivalences of block algebras, such as Morita equivalences and source-algebra equivalences. We refer to \cite{LinckBook, ThevenazBook} for more details.\\

\subsection{Trivial source lattices}
An indecomposable $RG$-lattice $M$ with vertex $Q\leq G$ is called a \emph{trivial source $RG$-lattice} if the trivial $RQ$-lattice $R$ is a source of $M$.  Equivalently,  an indecomposable $RG$-lattice is a trivial source $RG$-lattice if and only if it is a direct summand of a permutation $RG$-lattice, and hence such lattices are also sometimes called \emph{$p$-permutation} $RG$-lattices. We adopt the convention that  \emph{trivial source $RG$-lattices} are indecomposable by definition.
\par
It is well-known that any trivial source $kG$-module $M$ is liftable to an $\cO G$-lattice (see e.g. \cite[Corollary 3.11.4]{BensonBookI}). More accurately, in general, such modules afford several lifts, but there is a unique one amongst these which is a trivial source $\cO G$-lattice. We denote this trivial source lift by $\wh{M}$ and by $\chi_{\wh{M}}$ the ordinary character afforded by $\wh{M}$, that is the character of $K\otimes_{\cO}\wh{M}$.
Character values of trivial source lattices have the following properties.

\begin{lem}[{}{\cite[Lemma II.12.6]{LandrockBook}}]\label{lem:tscharacters}
Let $M$ be a trivial source $kG$-module with character $\chi_{\wh{M}}$,  and let $x$ is a $p$-element of $G$. Then:
\begin{enumerate}
\item[\rm(a)] $\chi_{\wh{M}}(x)$ equals the number of indecomposable direct summands of $\Res^{G}_{\langle x \rangle}(M)$ isomorphic to the trivial $k{\langle x \rangle}$-module. In particular,  $\chi_{\wh{M}}(x)$ is a non-negative integer.
\item[\rm(b)] $\chi_{\wh{M}}(x)\neq 0$ if and only if $x$ belongs to some vertex of $M$.
\end{enumerate}
\end{lem}

\noindent 

\noindent Furthermore, following the terminology of \cite[Definition 4.1.10]{HissLux}, we will call  an indecomposable $RG$-lattice $M$ with vertex $Q\leq G$ a \emph{cotrivial source $RG$-lattice} if the $RQ$-lattice $\Omega(R)$ is a source of $M$. Cotrivial source lattices have similar properties to trivial source lattices. In fact, an $RG$-lattice $M$ is a cotrivial source $RG$-lattice if and only if  $\Omega^{-1}(M)$ is a trivial source $RG$-lattice. \\

%

\subsection{Endo-permutation modules}
Endo-permutation modules over  a finite $p$-group $P$ were introduced by E.C. Dade in \cite{DadeI&II}. 
A $kP$-module $M$ is called \emph{endo-permutation} if its  $k$-endomorphism algebra $\End_k(M)$ is a permutation $kP$-module, where $\End_k(M)$ is endowed with its natural $kP$-module structure via the action of $G$ by conjugation, i.e.: 
$$\lconj{g}{\phi}(m)=g\cdot\phi(g^{-1}\cdot m)\quad \forall\,g\in P, \forall\, \phi\in \End_k(M)\text{ and }\forall\,m\in M \,.$$
Furthermore, an endo-permutation $kP$-module $M$ is said to be \emph{capped} if it has an indecomposable direct summand with vertex $P$. 

 In this case, $M\cong \Capp(M)^{\oplus m}\oplus X$, where $m\geq1$ is an integer, all the indecomposable direct summands of $X$ have vertices strictly contained in $P$, and $\Capp(M)$, called the \emph{cap} of $M$, is, up to isomorphism, the unique indecomposable direct summand of $M$ with vertex $P$.  We denote by $\dade(P)$ the  Dade group of $P$, which we see as  the set of isomorphism classes of indecomposable capped endo-permutation $kP$-modules with composition law induced by the tensor product over $k$, that is 
$$M+N:=\Capp(M\otimes_k N)\,.$$
Since $\End_k(M)\cong M^*\otimes_k M$, obviously the identity element is the trivial module and the inverse of $M$ is its $k$-dual $M^*$. Moreover, if $Q\leq P$, then we denote by $\Omega_{P/Q}$ the relative Heller operator with respect to $Q$. In other words, if $M$ is a $kP$-module, then $\Omega_{P/Q}(M)$ is the kernel of a relative $Q$-projective cover of  $M$.  (See \cite{Kno78, TheRelProj}.) 
With this notation the usual Heller operator is $\Omega=\Omega_{P/\{1\}}$ and $\Omega_{P/Q}(k)$ is  the kernel of the augmentation map $k[P/Q]\lra k$  (mapping every element of the basis $P/Q$ of the $kP$-permutation module $k[P/Q]$ to $1$). It follows easily that $M$ is an endo-permutation $kP$-module if and only if $\Omega_{P/Q}(M)$ is an endo-permutation $kP$-module. 
For further details on endo-permutation modules we refer the reader to the survey \cite[\S3-\S4]{TheSurvey} and the references therein.\\

\vspace{6mm}
\section{Cyclic Blocks: quoted results}\label{sec:cyclicblocks}

The theory of blocks with cyclic defect groups essentially goes back to \cite{Dad66} and was then developed by many different authors. We recall here the main results of this theory which we will be using throughout. 

\subsection{Notation}
From now on, unless otherwise stated,  we let ${\bf B}$ denote a block of $kG$ with cyclic defect group $D\cong C_{p^n}$ with $n\geq 1$.
For $0\leq i\leq n$, we denote by $D_i$ the unique cyclic subgroup of  order $p^i$ and we set $N_i:=N_G(D_i)$. 
Therefore we have the following chain of subgroups:
$$\{1\}=D_0 < D_1<\ldots< D_n=P_n\leq N_n\leq N_{n-1}\leq \ldots \leq N_1\leq G$$
and 
$$\smallskip  C_G(D)\leq C_G(D_{n-1})\leq \ldots\leq C_G(D_1)\leq C_G({D_0})=G\,.$$
We let $e$ denote the inertial index of $\bB$ and $m:=\frac{|D|-1}{e}$ be  the \emph{exceptional multiplicity} of~$\bB$. Then $e\mid p-1$, there are $e$ simple $\bB$-modules $S_1,\ldots,S_e$, and $e+m$ ordinary irreducible characters. We write
$$\Irr(\bB)=\Irr'(\bB)\sqcup\{\chi_{\lambda} \mid \lambda\in\Lambda\}\,,$$
where $|\Irr'(\bB)|=e$ and $\Lambda$ is an index set with  $|\Lambda|:=m$. If $m>1$, the characters $\{\chi_{\lambda} \mid \lambda\in\Lambda\}$ denote the exceptional characters of $\bB$, which  all restrict in the same way to the $p$-regular conjugacy classes of~$G$ and $\Irr'(\bB)$ consists of the non-exceptional characters of $\bB$, which are $p$-rational. 
We set $\chi_{\Lambda}:=\sum_{\lambda\in\Lambda}\chi_{\lambda}$ and $\Irr^{\circ}(\bB):=\Irr'(\bB)\sqcup\{\chi_{\Lambda}\}$. 
Furthermore, to a  block $\bB$ with cyclic defect groups are associated two important graphs:   the Brauer tree  $\sigma(\bB)$ of $\bB$ and the stable Auslander-Reiten quiver (or in short the stable AR-quiver) $\Gamma_s(\bB)$ of $\bB$.\\

\subsection{The Brauer tree}
The vertices of $\sigma(\bB)$ are labelled by the ordinary characters in $\Irr^{\circ}(\bB)$ and  the edges of $\sigma(\bB)$ are labelled by the simple $\bB$-modules $S_1,\ldots,S_e$. If $m>1$  the vertex corresponding to $\chi_{\Lambda}$ is called the \emph{exceptional vertex} and is indicated with a black circle in the drawings of $\sigma(\bB)$. We call \emph{leaf}  a vertex of $\sigma(\bB)$ with valency~$1$ or the edge adjacent to it. Hence the leaves of $\sigma(\bB)$ correspond to the simple liftable $\bB$-modules.  Furthermore, we assume that $\sigma(\bB)$ is given with a planar embedding, determined by specifying, for each vertex of $\sigma(\bB)$, a cyclic ordering of the edges adjacent to this vertex. We use the convention that in a drawing of $\sigma(\bB)$ in the plane, the successor of an edge is the counter-clockwise neighbour of this edge.  
For more detailed information relative to Brauer trees we refer the reader to \cite[\S 17]{AlperinBook} and \cite[Chapters 1 \& 2]{HissLux}. \\

\subsection{The stable Auslander-Reiten quiver} It is well-known that $\Gamma_s(\bB)$ is a finite tube of type $(\IZ/e\IZ)A_{p^{n}-1}$, where the Auslander-Reiten translate is given by $\Omega^{2}$. See e.g. \cite[Theorem 6.5.5]{BensonBookI}. \\
We say  that an indecomposable $\bB$-module  $X$ \emph{lies on the boundary}  of $\Gamma_s(\bB)$ if the middle term of the Auslander-Reiten sequence terminating at $X$ has exactly one non-projective indecomposable summand.  
If $D=C_2$, then $\Gamma_s(\bB)$ consists of a single vertex. If $D\neq C_2$ the boundary of a finite tube $(\IZ/e\IZ)A_{p^{n}-1}$ consists of the disjoint union of two $\Omega^2$-orbits of modules
$$\{\Omega^{2i}(X)\mid 0\leq i\leq e-1\}\text{ and }\{\Omega^{2i+1}(X)\mid 0\leq i\leq e-1\}\,,$$
where $X$ can always be chosen to be a non-exceptional leaf of $\sigma(\bB)$ (c.f. \cite[Proposition~3.3]{BessenrodtCyclic}). For a detailed treatment of the stable  AR-quiver of cyclic blocks we refer to  \cite[\S6.5]{BensonBookI} and  \cite{BessenrodtCyclic}.  Moreover, \cite{BleChi} determines  the distances to the boundaries of  $\Gamma_s(\bB)$ of any given indecomposable $\bB$-module.\\

\subsection{The Brauer correspondent of $\bB$ in $N_1$}\label{ssec:BrauerCorresp}  Throughout, we let  $\bb$ denote the Brauer correspondent of $\bB$ in~$N_1$. 
The structure of $\bb$ is given by Lemma~\ref{lem:complength} below. The inertial index of $\bb$ equals that of $\bB$, that is $e$, and the exceptional multiplicity of $\bb$ equals $m$. The stable AR-quiver of $\bb$ is again a tube of the form $(\IZ/e\IZ)A_{p^{n}-1}$. The Brauer tree $\sigma(\bb)$ of $\bb$ is a star with $e$ leaves, and  with exceptional vertex at its centre if $e,m>1$. If $e>1$ and $m=1$, then  we may assume that we have chosen the labelling of $\Irr(\bb)$ such that $\Irr'(\bb)$ labels the leaves of $\sigma(\bb)$. Therefore all $e$ simple $\bb$-modules
are liftable and if $e>1$ any of their lifts affords the $K$-character in $\Irr'(\bb)$ labelling the corresponding leaf of $\sigma(\bb)$. 
\par 
The Green correspondence with respect to $(G,N_1;D)$, which we denote by $f$ (upwards) and $g$ (downwards), commutes with the Brauer correspondence and with the Heller operator. Therefore, $f$~and~$g$ induce  $\Omega^2$-equivariant graph isomorphisms between $\Gamma_s(\bB)$ and $\Gamma_s(\bb)$ preserving vertices and sources (c.f. \cite[Theorem 6.6.5 and its proof]{BensonBookI}). More precisely, if $X$ is an indecomposable non-projective $\bB$-module with vertex contained in $D$, then $f(X)$ belongs to $\bb$, and $X$ and $f(X)$ lie at the same distance to the boundary of the stable AR-quiver. 
Likewise, if $Y$ is an indecomposable non-projective $\bb$-module with vertex in $D$, then  $g(Y)$ belongs to~$\bB$, and $Y$ and $g(Y)$ lie at the same distance to the boundary of the stable AR-quiver.\\

\vspace{2mm}
\subsection{PIMs and hooks}\label{ssec:PIMs}
Cyclic blocks being Brauer graph algebras (with respect to the Brauer tree), the structure of the PIMs of $\bB$,  can be described as follows (see e.g. \cite[\S 4.18]{BensonBookI}). If $S_j$ is a simple $\bB$-module, then its projective cover $P_{S_j}$ is of the form
$$P_{S_j}=\boxed{\begin{smallmatrix} S_j \\ Q_a\oplus\, Q_b \\ S_j\end{smallmatrix}}\,,$$
where $S_j=\soc(P_{S_j})=\head(P_{S_j})$ and the heart of $P_{S_j}$ is  $\rad(P_{S_j})/\soc(P_{S_j})=Q_a\oplus Q_b$ for two uniserial (possibly zero) $\bB$-modules $Q_a$ and $Q_b$. 
Furthermore, if the edge of $\sigma(\bB)$ corresponding to $S_j$ is 
$$
\begin{tikzcd}
\cdots\,\underset{\chi_a}{{\Circle}} \arrow[r, dash,"S_j"] & \underset{\chi_b}{{\Circle}} \,\cdots
\end{tikzcd}
$$
then the descending composition series of $Q_i$ ($i\in\{a,b\}$) 
is read off $\sigma(\bB)$ as follows: if the edges around $\chi_i$ are cyclically ordered $j,j_1,j_2,\ldots,j_r,j$, then $Q_i$ has descending composition factors
$$
\begin{cases}
 S_{j_1}, \ldots,S_{j_r}     & \text{if $\chi_i$ is not the exceptional vertex of }\sigma(\bB), \\
 S_{j_1}, \ldots,S_{j_r},S_j, S_{j_1}, \ldots,S_{j_r},S_j,\ldots,\ldots, S_{j_r}   & \text{if $\chi_i$ is the exceptional vertex of }\sigma(\bB),
\end{cases}
$$
where in the second case $S_{j_1}, \ldots,S_{j_r}$ appear $m$ times and $S_j$ appears $m-1$ times. Moreover the projective indecomposable character corresponding to $P_{S_j}$ is $\Phi_{S_j}=\chi_{a}+\chi_{b}$.
The PIMs of~$\bB$ are precisely the trivial source $\bB$-module with vertex $D_0=\{1\}$.\\ 

Next we recall that Green's walk around the Brauer tree \cite{GreenWalk} provides us with a description of the $\bB$-modules forming the boundary of $\Gamma_s(\bB)$ and their ordinary characters, that is the 
set of the Heller translates
$$\{\Omega^{i}(S)\mid 0\leq i\leq 2e-1\}$$
of a fixed simple $\bB$-module $S$ corresponding to a non-exceptional  leaf of $\sigma(\bB)$.

These modules  are called \emph{hooks} (and \emph{cohooks}) in \cite{BleChi} and in \cite{HN12}'s classification theorem of liftable $\bB$-modules (see Theorem~\ref{thm:classific_liftable}). More precisely, following the terminology used in \cite[\S 2.3]{BleChi}, the uniserial modules of the form 
$$H_a:=\boxed{\begin{smallmatrix}S_j\\Q_a\end{smallmatrix}}\qquad\text{ and }\qquad H_b:=\boxed{\begin{smallmatrix}S_j\\Q_b\end{smallmatrix}}$$
for a simple $\bB$-module $S_j$ are called the \textit{hooks} of $\bB$ and the uniserial modules
of the form 
$$C_a:=\boxed{\begin{smallmatrix}Q_a\\S_j\end{smallmatrix}}\qquad\text{ and }\qquad C_b:=\boxed{\begin{smallmatrix}Q_b\\S_j\end{smallmatrix}}$$
are called the \textit{cohooks} of $\bB$. By the above description of the PIMs, clearly  $\Omega(H_a)=C_b$, $\Omega(H_b)=C_a$, and  each hook is a cohook and conversely. Furthermore, if $e>1$ any lift of $H_a$ and $C_a$ affords the character $\chi_a$ and any lift of $H_b$ and $C_b$ affords the character $\chi_b$ (see also Theorem~\ref{thm:classific_liftable}(c)(2)).  Hence for simplicity we will say that $\chi_a$ (resp.~$\chi_b$) is afforded by the hooks $H_a$ and $C_a$ (resp.  $H_b$ and~$C_b$). 
Hooks are not trivial source modules in general, but they are essential to parametrise the position of the trivial source modules in~$\Gamma_s(\bB)$.

\vspace{6mm}
\section{Location of the trivial source modules in the stable AR-quiver}\label{sec:location}

The aim of this section is to determine the distance in $\Gamma_s(\bB)$ of the trivial source modules with vertex $D_i$ ($1\leq i\leq n$) to one of the boundaries. In order to avoid technicalities, unless otherwise stated throughout this section we will assume that $D\ncong C_2$, and we treat the case $D\cong C_2$ in $\S\ref{ssec:char2}$ below. Thus, $\Gamma_s(\bB)$ has two disjoint boundaries. 

It is clear that an $\Omega^2$-orbit of modules in  $\Gamma_s(\bB)$ has a common vertex $D_i\leq D$ and a common $kD_i$-source because indecomposable $kD_i$-modules are $\Omega$-periodic of period $2$.  It is also clear that there is precisely one $\Omega^2$-orbit of modules in $\Gamma_s(\bB)$ with vertex $D_i$ and trivial source (see e.g. \cite[p. 257, l. 1-3]{BessenrodtCyclic}). Therefore, we need to determine how many rows away from the boundary these orbits lie: this is what we call \emph{the distance to the boundary}. The second issue we need to deal with is the problem of determining from which of the two boundaries we start.

\vspace{2mm}
\subsection{The case in which $D_1$ is normal.}

\begin{lem}\label{lem:complength}
Let $H$ be an arbitrary finite group and let $\bA$ be a block of $kH$ with a non-trivial cyclic defect group $D\cong C_{p^n}$ ($n\geq 1$) and  inertial index $\tilde{e}$. If $D_1$ is normal in $H$, then the following holds. 
\begin{enumerate}
\item[\rm(a)]  The Brauer tree of $\bA$ is a star with $\tilde{e}$ edges and  exceptional vertex at its center. In particular $\bA$ is Morita equivalent to $k[D\rtimes C_{\tilde{e}}]$, where $C_{\tilde{e}}$ acts faithfully on~$D$. The PIMs of $A$ are uniserial, and therefore so is any indecomposable $\bA$-module, and in particular a quotient of a PIM. Moreover $\Gamma_s(\bA)$ is a tube of the form $(\IZ/\tilde{e}\IZ)A_{p^{n}-1}$.
\item[\rm(b)] The $\tilde{e}$ simple $\bA$-modules all have the same $k$-dimension and form one  of the two boundary $\Omega^2$-orbits of $\Gamma_s(\bA)$. We call this  $\Omega^2$-orbit the 1st row of $\Gamma_s(\bA)$.
\item[\rm(c)] The indecomposable $\bA$-modules forming the $\Omega^2$-orbit corresponding to the $i$-th row of $\Gamma_s(\bA)$ are all uniserial of  length $i$, and hence of the same $k$-dimension.
\end{enumerate}
\end{lem}

\begin{proof}{\ }
\begin{enumerate}
\item[\rm(a)]  See \cite[\S 17, Theorem 2]{AlperinBook} and \cite[Lemma 6.5.2 and Proposition 6.5.4]{BensonBookI}. 
\item[\rm(b)]  See \cite[Proposition 6.5.4]{BensonBookI} and \cite[Theorem 3.6(i) and (ii)]{BessenrodtCyclic}. 
\item[\rm(c)]  See \cite[p. 256]{BessenrodtCyclic}.  This is in fact a standard argument of Auslander-Reiten theory: since the irreducible maps determining the arrows of $\Gamma_s(\bA)$ are either injective or surjective, an arrow coming out of a simple module is injective and an arrow pointing into a simple module is surjective. It follows that the modules lying on the second row of $\Gamma_s(\bA)$ all have composition length 2, and by induction the modules lying on the $i$-th row of $\Gamma_s(\bA)$ all have composition length $i$.
\end{enumerate}
\end{proof}

\noindent In particular, under the hypothesis of Lemma~\ref{lem:complength},  the distance from an indecomposable $\bA$-module  $Y$ to  the boundary $\Omega^2$-orbit of $\Gamma_s(\bA)$ containing the simple $\bA$-modules is given by $\ell(Y)-1$.

\vspace{2mm}
\subsection{From $\bB$ to  $\bb$} 
In view of the above the two boundary $\Omega^2$-orbits of $\bB$ consist precisely of the hooks of $\bB$, and their Green correspondents are the hooks of $\bb$, i.e. the simple $b$-modules and their 1st Heller-translates.
\par
The next lemma tells us that the character values of the hooks at the non-trivial elements of~$D$ are determined by the character values of the simple $\bb$-modules. We note that this result is stated without proof in \cite[Lemma 2.5]{KK10}.

\begin{lem}\label{lem:CharValues}
Let $S$ be a simple $\bb$-module and let $g(S)$ be its Green correspondent. Let $\chi_S\in\Irr'(\bb)$ be the $K$-character labelling the leaf of $\sigma(\bb)$ corresponding to $S$ and let $\chi_{g(S)}\in\Irr^{\circ}(\bB)$ be the $K$-character afforded by the hook $g(S)$.  Then
$$\chi_{g(S)}(x)=\chi_{S}(x)$$
 for every  $x\in D\setminus\{1\}$. 
\end{lem}

\begin{proof} Since $D_1$ is a trivial intersection subgroup of $G$, by \cite[\S17, Theorem~3]{AlperinBook}, we have that
$$\Ind_{N_1}^{G}(S)\cong g(S)\oplus (\text{projective summands}).$$
Therefore 
$$\Ind_{N_1}^{G}(\chi_{S})=\chi_{g(S)}+(\text{projective characters})$$
and,  as projective characters vanish at non-trivial $p$-elements, we obtain that 
$$\big(\Ind_{N_1}^{G}(\chi_{S})\big)(x)=\chi_{g(S)}(x)\quad\quad\forall x\in D\setminus\{1\}.$$
Moreover, 
$$\big(\Ind_{N_1}^{G}(\chi_{S})\big)(x)=\sum_{i=1}^{r}\overset{\cdot}{\chi_{S}}(x_ixx_i^{-1})\,$$
where $x_1:=1_G,\ldots,x_r$ ($r\in\IN$) is a set of representatives of the left cosets of $N_1$ in $G$ and $\overset{\cdot}{\chi_{S}}(g)=\chi_{S}(g)$ if $g\in N_1$ whereas $\overset{\cdot}{\chi_{S}}(g)=0$ if $g\in G\setminus N_1$.
In addition, by the proof of \cite[\S17, Theorem~3]{AlperinBook} the non-trivial subgroups of the groups of the form $N_1\cap x_iDx_i^{-1}$ are not $N_1$-conjugate to a subgroup of $D$ provided $x_i\in G\setminus N_1$. 
 Therefore it follows from Green's Theorem on Zeros of Characters (see \cite[(19.27)]{CR1}) that
$\overset{\cdot}{\chi_{S}}(x_ixx_i^{-1})=0$ for each $i\neq 1$ because the simple $\bb$-modules are $D$-projective, and hence so are their lifts. Thus
$$\chi_{g(S)}(x)=\big(\Ind_{N_1}^{G}(\chi_{S})\big)(x)=\chi_{S}(x)\quad\quad\forall x\in D\setminus\{1\}.$$
\end{proof}

\begin{cor}\label{cor:positivity/negativity}
Let  $\chi\in \Irr^{\circ}(\bB)$, let $u$ be a generator of $D_1$, let $d$ be the dimension of the simple $\bb$-modules. 
\begin{enumerate}
\item[\rm(a)] If $\chi$ is the character afforded by a hook of $\bB$ whose Green correspondent is a simple $\bb$-module, then  $\chi(u)=d$.
\item[\rm(b)] If $\chi$ is the character afforded by a hook of $\bB$ whose Green correspondent is of the form $\Omega(S)$ for a simple $\bb$-module $S$, then  $\chi(u)=-d$. 
\end{enumerate}
\end{cor}

\begin{proof}
By \S\ref{ssec:PIMs} any $\chi\in \Irr^{\circ}(\bB)$ is afforded by a hook $H$ of $\bB$. Therefore, there is a simple $\bb$-module $S$ such that either $H=g(S)$ or $H=g(\Omega(S))$.\\
If $H=g(S)$ and $\chi_S\in\Irr'(\bb)$ is the character labelling the leaf of $\sigma(\bb)$ defined by $S$, then by Lemma~\ref{lem:CharValues} we have $\chi(u)=\chi_{S}(u)$. Moroever, by Clifford theory $D_1$ acts trivially on the simple $\bb$-modules, hence $\chi_{S}(u)=d$ and  (a) follows. \\
If $H=g(\Omega(S))$, then by Green's walk around $\sigma(\bB)$ there is an edge of $\sigma(\bB)$ adjacent to $\chi$  whose other extremity, say labelled by $\chi'\in\Irr^{\circ}(\bB)$, is such that $\chi'$ is afforded by the hook $\Omega^{-1}(H)$. Then $f(\Omega^{-1}(H))=S$ and $\chi'(u)=d$ by (a). Now, $\chi+\chi'$ being a projective indecomposable character, it vanishes at non-trivial $p$-elements and it follows that $\chi(u)=-d$.
\end{proof}

As above, let $u$ be a generator of $D_1$. In consequence, we will call a vertex of $\sigma(\bB)$ corresponding to a character $\chi\in \Irr^{\circ}(\bB)$ such that $\chi(u)> 0$  \emph{positiv} and write $\chi>0$.  We will call a vertex of $\sigma(\bB)$ corresponding to a character $\chi\in \Irr(\bB)$ such that $\chi(u)< 0$  \emph{negativ} and we write $\chi<0$. Because the projective indecomposable characters of $\bB$ are of the form $\Phi_{S_j}=\chi_a+\chi_b$ ($1\leq j\leq e$), where $\chi_a$ and $\chi_b$ label the end vertices of the edge corresponding to the simple $\bB$-module $S_j$ (see \S\ref{ssec:PIMs}), a positive vertex of~$\sigma(\bB)$ can only be linked to a negative vertex and conversely.
\par
It follows that one of the boundary $\Omega^2$-orbit of $\Gamma_s(\bB)$ consists of the hooks with positive ordinary characters, and the other boundary $\Omega^2$-orbit of $\Gamma_s(\bB)$ consists of the hooks with negative characters.  Therefore, given a non-projective indecomposable $\bB$-module $X$,  we define its \emph{(positive) distance} $d^{+}(X)$ (resp. its \emph{negative distance}  $d^{-}(X)$) to be the length of a shortest path in $\Gamma_s(\bB)$ to the boundary consisting of the positive (resp. negative) hooks.  Clearly 
$$d^{+}(X)+d^{-}(X)=(p^{n}-1)-1=em-1\,.$$

\begin{cor}\label{cor:d+GreenCorresp}
Let  $X$ be  a non-projective indecomposable $\bB$-module. Then 
$$d^+(X)=d^{+}(f(X))=\ell(f(X))-1\,.$$
\end{cor}

\begin{proof}
By definition and Corollary~\ref{cor:positivity/negativity} we have  $d^+(X)=d^{+}(f(X))$. By Lemma~\ref{lem:complength}(c) we have ${d^{+}(f(X))=\ell(f(X))-1}$. 
\end{proof}

\vspace{2mm}
\subsection{Reduction to the centraliser $C_{G}(D_1)$}

Let $\bc$ be a block of $C_{G}(D_1)$ covered by $\bb$ and let $T:=T(\bc)$ be the inertia group of $\bc$ in $N_1$. 
Let $\bb'$ be the unique block of $T$ covering $\bc$ such that $\bb'^{N_1}=\bb$ (i.e. the Fong-Reynolds correspondent of~$\bb$). 
Then $D$ is also a defect group of $\bb'$ and~$\bc$, $|T:C_G(D_1)|=e$, the inertial index of $\bb'$ is~$e$, whereas $\bc$ is nilpotent. See \cite[\S 17-19]{AlperinBook}.
By  Lemma~\ref{lem:complength} the blocks $\bb'$ and $\bc$ are also uniserial,  $\Gamma_s(\bb')\cong(\IZ/e\IZ)A_{p^{n}-1}$ and $\Gamma_s(\bc)\cong A_{p^{n}-1}$.\\

We reduce the problem of determining the distances  in $\Gamma_s(\bB)$ of the trivial source modules  to one of the boundaries to the problem of determining  the composition lengths of the trivial source  $\bc$-modules through the following steps:

$$
\begin{tikzcd}
G \arrow[d, dash]  & & \bB-\mbox{mod}  \arrow[d, xshift=0.7ex, "{\text{\tiny Green correspondence (stable equivalence of Morita type)}}"] \\  
N_1 \arrow[d, dash] & & \bb-\text{mod}  \arrow[u, xshift=-0.7ex]   \arrow[d, xshift=0.7ex, "{\text{\tiny source algebra equivalence (induced by $\Ind_{T}^{N_1}$)}}"] \\ 
T \arrow[d, dash, "e"] & &  \bb'-\text{mod}  \arrow[u, xshift=-0.7ex] \\
C_{G}(D_1) & &  \bc-\text{mod}  \arrow[u, dashrightarrow, "\text{\tiny Clifford theory (induced by $\Ind_{C_G(D_1)}^{T}$)}"'] 
\end{tikzcd}
$$
More accurately, if $V$ denote the unique simple $\bc$-module, then the following is well-known:
\begin{itemize}
\item[\rm 1.]  $\Ind_{C_G(D_1)}^{T}(V)$ has $e$ non-isomorphic direct summands $V_1,\ldots, V_e$, which constitute a complete set of representatives of the isomorphisms classes of the simple $\bb'$-modules. (See \cite[\S19, Lemma~4]{AlperinBook}.) It follows that for every $1\leq i\leq |D|-1$, induction from $C_G(D_1)$ to $T$ of a  uniserial indecomposable module of length $i$  is a direct sum of $e$ non-isomorphic uniserial indecomposable modules of length $i$.  Hence lengths, vertices and sources are preserved by $\Ind_{C_G(D_1)}^{T}$ between $\bc$ and $\bb'$. 
\item[2.] By the Fong-Reynolds theorem the induced modules $\Ind_{T}^{N_1}(V_i)$ ($1\leq i\leq e$)  then form a complete set of representatives of the isomorphisms classes of  the simple $\bb$-modules.  In fact, by \cite[(9.5)]{LinckThese} induction from $T$ to $N_1$ induces a source-algebra equivalence between $\bb'$ and $\bb$. Therefore, lengths, vertices and sources are also preserved by $\Ind_{T}^{N_1}$ between $\bb'$ and $\bb$. 
\end{itemize} 

\begin{cor}\label{cor:d+c}
Let $M$ be the unique  trivial source $\bc$-module with vertex $D_i$ ($1\leq i\leq n$). Then the indecomposable direct summands of $\Ind_{C_G(D_1)}^{N_1}(M)$ are precisely the  trivial source $\bb$-modules  with vertex $D_i$, and for any such direct summand $Y$ we have $d^+(Y)=\ell(M)-1$.  
\end{cor}

\begin{proof}The first claim follows directly from 1. and 2. above. Moreover, as $D_1$ is normal in $N_1$, $T$ and $C_G(D_1)$, it follows from Lemma~\ref{lem:complength}(c) and the above that 
$${d^{+}(Y)=\ell(Y)-1=\ell(M)-1}\,.$$
\end{proof}

\vspace{2mm}

\subsection{Reduction to a source algebra of  $\bc$}\label{ssec:sourceAlgebra}
We now describe how to use a source algebra of the block $\bc$ in order to recover the trivial source modules of this block. 
Puig's characterisation of nilpotent blocks,  
see \cite[Corollary 8.11.11]{LinckBook}, 
states that any source algebra $A$ of $\bc$ is isomorphic to 
$$S\otimes_k kD$$ 
as interior $D$-algebra, where $S:=\End_k(W)$ for an indecomposable capped endo-permutation $kD$-module $W$. Moreover, $W$ is a source of the unique simple $\bc$-module $V$, and also a source of the simple $\bb$-modules.   Recall  that as  $D_1\trianglelefteq C_G(D_1)$ it follows from Clifford theory  that $D_1$ acts trivially on $V$, hence also on $W$.

\noindent Then, we have two Morita equivalences:
$$
\begin{tikzcd}
kD-\text{mod}  \arrow[r, leftrightarrow,  "\sim_M"]   &  A -\text{mod}   \arrow[r, leftrightarrow, "\sim_M"]   &     \bc-\text{mod} 
\end{tikzcd}
$$
The first one is obtained by tensoring over $k$ with $W$ viewed as an $S$-module. In other words, an arbitrary indecomposable $A$-module is of the form $W\otimes_k U$, where $U$ is an indecomposable $kD$-module. For the second one let $i\in \bc^D$ be a source idempotent of $\bc$ such that $A=ikGi$. Then the $(\bc,A)$-bimodule $\bc i$ and the $(A,\bc)$-bimodule $i\bc$ realise a Morita equivalence between $A$ and $\bc$, where an indecomposable $\bc$-module $M$ corresponds to the $A$-module $iM$. See e.g. \cite[(38.2)]{ThevenazBook}.
\par
Furthermore, vertices and sources of the indecomposable $\bc$-modules are detected from the source algebra $A$ as follows. If $N$ is an indecomposable $A$-module, then a minimal subgroup $Q\leq D$ such that $N$ is isomorphic to a direct summand of $A\otimes_{kQ}Z$ for some indecomposable direct summand $Z$ of $\Res^{A}_{kQ}(N)$ is called a vertex of $N$ and $Z$ is called a $kQ$-source of~$M$.
(See \cite[\S6.4]{LinckBook}). 
The pair $(Q,Z)$ is then also a vertex-source pair of the indecomposable $\bc$-module $\bc i\otimes_A N$ by \cite[Theorem~6.4.10]{LinckBook}. The converse does not hold in general, but it is true for the block~$\bc$, since it is a nilpotent cyclic block. (See \cite[Remark 6.5]{LinckKlein4}.) Thus, taking $Z=k$ allows us to recover the trivial source $\bc$-modules from $kD$.

\begin{lem}\label{lem:tsc}
Let $M$ be the unique indecomposable trivial source $\bc$-module with vertex \linebreak $1<Q\leq D$. Then $M$ corresponds to the $kD$-module
$$U_Q(W):=\left(\Ind_Q^D\circ\Capp\circ\Res^D_Q\right)\!(W)$$
under the above Morita equivalences, and $\ell(M)=\dim_k U_Q(W)$. 
\end{lem}

\begin{proof}
We need to determine the unique  indecomposable  $kD$-module $U_Q:=U_Q(W)$ such that \linebreak $W\otimes_k U_Q=iM$ has vertex $Q$ and a trivial $kQ$-source.\\
Since $W$ is a capped endo-permutation $kD$-module, its restriction to $Q$ is again capped endo-permutation and 
$$\Res^D_Q(W)=\Capp(\Res^D_Q(W))^{\oplus a}\oplus X\,,$$
where $a\geq 1$ is an integer and all the direct summands of $X$ have vertices strictly contained in~$Q$. It follows that  as $kD$-modules we have 
\begin{equation*}
\begin{split}
 A\otimes_{kQ} k \cong S\otimes_k kD\otimes_{kQ} k& \cong (W\otimes_k W^*)\otimes_k kD\otimes_{kQ}k \\
                                                          & \cong W\otimes_k(W^*\otimes_k\Ind_Q^D(k))\\
                                                          & \cong  W\otimes_k \Ind_Q^D(\Res^D_Q(W^*)\otimes_k k) \\
                                                          & \cong  W\otimes_k \Ind_Q^D(\Res^D_Q(W^*))   \\
                                                          & \cong  W\otimes_k  \left(\Ind_Q^D\circ\Capp\circ\Res^D_Q\right)\!(W^*)^{\oplus a}\oplus    W\otimes_k X  .                                                                 
\end{split}
\end{equation*}
Since $D$ is cyclic $W^{*}\cong W$ and $iM$ must correspond, up to isomorphism, to the unique indecomposable summand of $A\otimes_{kQ} k$ with vertex $Q$.
Hence  $U_Q=\left(\Ind_Q^D\circ\Capp\circ\Res^D_Q\right)\!(W)$. \\
The second claim follows from the above Morita equivalences. 
\end{proof}

\noindent Therefore, we need to compute the length of the $kD$-module $U_Q(W)$ for an arbitrary non-trivial $p$-subgroup $Q$ of $D$ and an arbitrary endo-permutation $kD$-module $W$, on which $D_1$ acts trivially.

\vspace{2mm}

\subsection{Restriction in the Dade group} \label{ssec:ep}
Endo-permutation modules over cyclic $p$-groups  were classified by Dade \cite{DadeI&II}.  A precise description of this classification is given by \cite[Exercise~(28.3)]{ThevenazBook} and \cite[Theorem 5.2]{TheSurvey}. Namely, the Dade group of a cyclic $p$-group $D\cong C_{p^n}$ with $p^n\geq 2$  is
$$\dade(D)=\langle \Omega_{D/D_i}(k) \mid 0\leq i \leq n-1 \rangle 
\cong 
\begin{cases}
 (\IZ/2)^{n-1}        & \text{if } p=2\,, \\
 (\IZ/2)^{n}           & \text{if }p\geq 3.
\end{cases}$$
(Notice that if $p=2$, then $\Omega_{D/D_{n-1}}(k)\cong k$, hence the missing $\IZ/2$ factor.)
In other words, the indecomposable capped endo-permutation  $kD$-module are, up to isomorphism, precisely the modules of the form
$$W_D(a_0,\ldots,a_{n-1}):=\Omega_{D/D_0}^{a_0}\circ\Omega_{D/D_1}^{a_1}\circ\cdots\circ\Omega_{D/D_{n-1}}^{a_{n-1}}(k)$$
with  $a_i\in\{0,1\}$ for each $0\leq i\leq n-1$. Moroever, we assume that $i_0<i_1<\ldots< i_s$ are the indices such that $a_{i_0}=\ldots=a_{i_s}=1$ and $a_i=0$ if $i\in\{0,\ldots,n-1\}\setminus\{i_0,\ldots, i_s\}$, and we set $s:=-1$ if $W_D(a_0,\ldots,a_{n-1})=k$. \\

If  $0\leq i\leq n-1$ and $M_b$ denotes the unique (up to isomorphism) indecomposable $kD$-module of dimension $1\leq b<p^{n-i}=\dim_k(k[D/D_i])$, then 
$$\dim_k(\Omega_{D/D_i}(M_b))= p^{n-i}-b\,.$$
Therefore an induction argument yields the following dimension formula:
$$\dim_k(W_D(a_0,\ldots,a_{n-1}))=\sum_{j=0}^{s}(-1)^j p^{n-i_j}+(-1)^{s+1}\,.$$
Now, restriction from $D$ to an arbitrary $p$-subgroup $D_i\leq D$ ($0\leq i\leq n$) and inflation from an arbitrary quotient $D/D_i$ ($0\leq i\leq n$)   induce
group homomorphisms
$$\Res^{D}_{D_i}:\dade(D)\lra \dade(D_i),\, M\mapsto \Capp(\Res^D_{D_i}(M))$$
and 
$$\Inf_{D/D_i}^{D}:\dade(D/D_i)\lra \dade(D),\, M\mapsto \Inf_{D/D_i}^{D}(M)\,.$$
\medskip 

In order to compute  $\Res^{D}_{D_i}(W_D(a_0,\ldots,a_{n-1}))$ it suffices to compute $\Res^{D}_{D_i}$ on the generators of~$\dade(D)$. 
Let $0\leq a\leq n-1$. Because we defined  $\Omega_{D/D_a}(k)$ to be a relative $D_a$-projective cover of the trivial $k[D/D_a]$-module, we have
$$\Omega_{D/D_a}(k)\cong\Inf_{D/D_a}^{D}(\Omega(k))\,.$$
This yields 
\begin{equation*}
\begin{split}
    \Res^{D}_{D_i}\left(\Omega_{D/D_a}(k) \right) & \cong  \Res^{D}_{D_i}\left( \Inf_{D/D_a}^{D}(\Omega(k)) \right)   \\
   									& \cong \left( \Inf_{D_i/D_i\cap D_a}^{D_i}\circ\Iso(\varphi^{-1})\circ\Res^{D/D_a}_{D_iD_a/D_a}\right) \left( \Omega(k) \right)   \\
									& \cong \begin{cases}  \left( \Inf_{D_i/D_i}^{D_i}\circ\Res^{D/D_a}_{D_a/D_a} \right) \!   \left( \Omega(k) \right)     & \smallskip \text{if } a\geq i\,, \\
									                                     \left( \Inf_{D_i/D_a}^{D_i}\circ\Res^{D/D_a}_{D_i/D_a} \right) \!   \left( \Omega(k) \right)  & \text{if } a<i \,, \end{cases}           
\end{split}
\end{equation*}
where $\varphi: \begin{tikzcd}[cramped, sep=small] D_i/D_i\cap D_a \arrow[r,"\cong"] & D_iD_a/D_a \end{tikzcd}$ is the canonical isomorphism. Therefore, taking the caps of these modules yields

\begin{equation*}
\begin{split}
   \Capp\circ \Res^{D}_{D_i}\left(\Omega_{D/D_a}(k) \right) & \cong \begin{cases}
   													   \left(\Capp\circ \Inf_{D_i/D_i}^{D_i} \right) \!\left( \Res^{D/D_a}_{D_a/D_a}\left(\Omega(k) \right)\right)     &\smallskip  \text{if } a\geq i, \\
									                                       \left( \Capp\circ  \Inf_{D_i/D_a}^{D_i} \right)  \! \left( \Res^{D/D_a}_{D_i/D_a}\left(\Omega(k)\right) \right)  & \text{if } a<i\,, \end{cases}  \\
									                                       &  \cong \begin{cases} 
									                                        \Inf_{D_i/D_i}^{D_i}(k)     & \text{if } a\geq i, \\
									                                         \left(\Capp\circ \Inf_{D_i/D_a}^{D_i}\right) \! \left( \Omega(\Res^{D/D_a}_{D_i/D_a} (k))\oplus (\text{projective}) \right)  & \text{if } a<i \,,
									                                        \end{cases}  \\
									                                        &  \cong \begin{cases}  k     & \text{if } a\geq i, \\
									                                        \Inf_{D_i/D_a}^{D_i}\left( \Omega(k) \right)  & \text{if } a<i \,,\end{cases}  \\
									                                         &  \cong \begin{cases}  k     & \text{if } a\geq i, \\
									                                        \Omega_{D_i/D_a}(k)  & \text{if } a<i \,.\end{cases}  
\end{split}
\end{equation*}
It follows that  $\ker(\Res^{D}_{D_i})=\langle \Omega_{D/D_a}(k) \mid i \leq a \leq n-1 \rangle$ and 
$$\Capp\circ \Res^{D}_{D_i}\left(W_D(a_0,\ldots,a_{n-1})\right)= \Omega_{D_{i}/D_0}^{a_0}\circ\Omega_{D_i/D_1}^{a_1}\circ\cdots\circ\Omega_{D_i/D_{i-1}}^{a_{i-1}}(k) =W_{D_i}(a_0,\ldots,a_{i-1})\,,$$
so that 
$$\dim_k\left(\Capp\circ \Res^{D}_{D_i}\left(W_D(a_0,\ldots,a_{n-1})\right) \right)=\sum_{0\leq i_j<i}(-1)^j p^{i-i_j}+(-1)^{|\{j\mid 0\leq i_j<i\}|}\,.$$

\vspace{2mm}
\subsection{The case $D\cong C_2$}\label{ssec:char2}

In characteristic $p=2$, it is always the case that $e=1$, since $e\mid p-1$.  If, moreover, the defect group $D$ is cyclic of order $2$, then $kD$ contains precisely two indecomposable modules: $k\cong\Omega(k)$ and $kC_2$. It follows that  the block $\bB$ contains precisely two indecomposable modules: a unique non-projective indecomposable module $S_1$, which is simple and the projective cover of $S_1$, which is uniserial of length~2. The former module is obviously a trivial source module with vertex $D=D_1$ and the latter module a trivial source module with vertex~$D_0=\{1\}$. 
Since $\Gamma_s(\bB)$ consists of a single vertex, we may set $d^{+}(S_1)=d^{-}(S_1)=0$. We note further that in this case, $\dade(D)=\{k\}$, so that the capped endo-permutation $kD$-module $W$ of the introduction is the trivial $kC_2$-module.

\vspace{6mm}
\section{The classification of the trivial source $\bB$-modules}\label{sec:mainresults} 

With the notation and the results of Section~\ref{sec:location}, we can state our two main theorems.
From now on $W$ denotes  the endo-permutation $kD$-module $W$ of \S\ref{ssec:sourceAlgebra} (resp. of \S\ref{ssec:char2} if $D\cong C_2$) which comes from Puig's description of the source algebra $A$ of the block $\bc$. Hence, it can be assumed that $W$ is by definition  a source of the unique simple $\bc$-module. \par
Notice that  by \cite[Theorem~2.7]{Linck96} $W$ is precisely  the module  $W$  of the introduction parametrising the source-algebra-equivalence class of the block $\bB$. Moreover, by \cite[Theorem~2.7]{Linck96}, as~$D_1$ acts trivially on~$W$, we have 
$$W=\Omega_{D/D_1}^{a_1}\circ\cdots\circ\Omega_{D/D_{n-1}}^{a_{n-1}}(k)$$
for intergers $a_1,\ldots,a_{n-1}\in\{0,1\}$.
In other words, with the notation of \S\ref{ssec:ep}, 
$$W=W_D(a_0,a_1,\ldots,a_{n-1})$$
with $a_0=0$.
Finally, for each $1\leq i\leq n$ we set $\ell_i:=\dim_k \Capp\left(\Res^D_{D_i}(W)\right)$. 

\enlargethispage{2mm}
\begin{thm}\label{thm:mainReform}\label{thm:mainA}
Let $\bB$ be a block of $kG$ with a non-trivial cyclic defect group $D\cong C_{p^n}$ ($n\geq 1$) and let~$W$ be the indecomposable capped endo-permutation $kD$-module associated to $\bB$.  Assume, moroever, that $W= W_D(0,a_1,\ldots,a_{n-1})$ and let $i_0<i_1<\ldots< i_s$ be the indices such that $a_{i_0}=\ldots=a_{i_s}=1$ and $a_i=0$ if ${i\in\{1,\ldots,n-1\}\setminus\{i_0,\ldots, i_s\}}$. 
Let $1\leq i\leq n$ and let $X$ be  a non-projective indecomposable trivial source $\bB$-module with vertex~$D_i$.  Then
$$d^+(X)=\ell_i\cdot p^{n-i}-1\,,$$
where $\ell_i=\sum_{0\leq i_j<i}(-1)^j p^{i-i_j}+(-1)^{|\{j\mid 0\leq i_j<i\}|}$\,.
\end{thm}

\begin{proof}If $D\cong C_2$, then the claim is straightforward from \S\ref{ssec:char2}. 
Else, it follows from Corollary~\ref{cor:d+GreenCorresp}, Corollary~\ref{cor:d+c} and Lemma~\ref{lem:tsc} that
$$d^+(X)=\ell(f(X))-1=\ell(M)-1=\dim_k (U_{D_i}(W))-1=\ell_i\cdot p^{n-i}-1\,,$$
where $M$ denotes  the unique trivial source $\bc$-module with vertex $D_i$ and 
$$U_{D_i}(W)=\left(\Ind_{D_i}^D\circ\Capp\circ\Res^D_{D_i}\right)\!(W)$$
is  the corresponding $kD$-module given by Lemma~\ref{lem:tsc}. Finally, \S\ref{ssec:ep} yields
$$\ell_i=\dim_k \Capp\left(\Res^D_{D_i}(W)\right)=\sum_{0\leq i_j<i}(-1)^j p^{i-i_j}+(-1)^{|\{j\mid0\leq i_j<i\}|}\,.$$
\end{proof}

As a corollary, we emphasise some cases in which the location of the trivial source $\bB$-modules in $\Gamma_s(\bB)$ is particularly easy to compute. 

\begin{cor}{\ }\label{cor:D1Dn}
\begin{enumerate}
\item[\rm(a)] If $X$ is a trivial source $\bB$-module with vertex $D_1$, then  $d^{+}(X)=p^{n-1}-1$.
\item[\rm(b)] If $X$ is a trivial source $\bB$-module with vertex $D=D_n$, then  $d^{+}(X)=\dim_k W-1\,.$
\item[\rm(c)] Let $X$ be a hook of $\bB$ and let  $\chi\in\Irr^{\circ}(\bB)$ be the character afforded by $X$. Then $X$ is a trivial source module if and only if  $W=k$ and $\chi>0$. 
\item[\rm(d)] If $\bB=\bB_0(kG)$ and $X$ is a trivial source $\bB$-module with vertex $D_i$ for $1\leq i\leq n$, then  $d^{+}(X)=p^{n-i}-1$.
\end{enumerate}
\end{cor}

\begin{proof}{\ }
\begin{enumerate}
\item[\rm(a)] Since $D_1$ acts trivially on $W$, the cap of $\Res^{D}_{D_1}(W)$ is the trivial $kD_1$-module, whence $\ell_1=1$ and Theorem~\ref{thm:mainReform} yields $d^{+}(X)=\ell_1\cdot p^{n-1}-1=p^{n-1}-1$. 
\item[\rm(b)] Since $D=D_n$, by definition  $\ell_n=\dim_k W$ and Theorem~\ref{thm:mainReform} yields $d^{+}(X)=\ell_n\cdot p^{n-n}-1=\dim_k W-1$.
\item[\rm(c)] Recall that hooks have vertex $D$ and form the boundary of $\Gamma_s(\bB)$. First assume that~$X$ is a trivial source module. Therefore, $\chi$ takes non-negative integer values at non-trivial $p$-elements by Lemma~\ref{lem:tscharacters}, so that $\chi>0$. Thus, it follows from (b) that $0=d^{+}(X)=\dim_k W-1$, so that $W=k$. 
Conversely, if $W=k$, then by (b) the trivial source modules with vertex $D$ lie at positive distance zero from the boundary of $\Gamma_s(\bB)$. Thus, as $\chi>0$,~$X$ must be a trivial source module. 
\item[\rm(d)] If $\bB$ is the principal block, then clearly $W=k$, the trivial $kD$-module. Therefore $\ell_i=1$ for each $1\leq i\leq n$ and the claim follows Theorem~\ref{thm:mainReform}. 
\end{enumerate}
\end{proof}

We can now use Theorem~\ref{thm:mainA}, the classification of the indecomposable liftable $\bB$-modules in Appendix~\ref{AppA}, as well as the computations  of the distances of the la{tt}er modules to the boundary of $\Gamma_s(\bB)$ given in  Appendix~\ref{AppB} in order to classify the trivial source $\bB$-modules with vertex $D_i$ for each $1\leq i\leq n$.

\begin{thm}\label{thm:mainB}
Let $\bB$ be a block of $kG$ with a non-trivial cyclic defect group $D\cong C_{p^n}$, $e$~simple modules, exceptional multiplicity $m=(p^n-1)/e$, Brauer tree $\sigma(\bB)$, and  let~$W$ be the indecomposable capped endo-permutation $kD$-module associated to $\bB$. For each $1\leq i\leq n$, let $D_i\leq D$ with $|D_i|=p^i$.
\begin{enumerate}
\item[\rm(a)] If $e=1$  and the Brauer tree of $\bB$ is
$
\begin{tikzcd}
\sigma(\bB) =  \overset{\chi_1}{{\Circle}}  \arrow[r, dash,"S_1"]  & \overset{\chi_{\Lambda}}{{\CIRCLE}} \,,
\end{tikzcd}
$ 
then the following holds:
\begin{itemize}
\item[\rm(i)] if $\chi_{1}>0$, then $\bB$ has a unique indecomposable trivial source module with vertex $D_i$, which is uniserial of length $\ell_i\cdot p^{n-i}$; 
\item[\rm(ii)] if $\chi_{1}<0$,  then $\bB$ has a unique indecomposable trivial source module with vertex $D_i$, which is uniserial of length $p^n-\ell_i\cdot p^{n-i}$.
\end{itemize}
\item[\rm(b)] Assume now that $e>1$. Then, an indecomposable $\bB$-module $X$ is a  trivial source module with vertex $D_i$ if and only if $X$ corresponds to one of the modules in {\rm\bf(1)--(7)} below.
In types {\rm\bf(2)}--{\rm\bf(7)}, $m>1$ holds. 
\begin{itemize}
  \item[\rm\bf(1)] The vertex is $D_i=D$, $W=k$, and $X$ is a hook affording a character $\chi\in\Irr^{\circ}(\bB)$ such that $\chi>0$. 
  \item[\rm\bf(2)] The module $X$ corresponds to the  path
  $$
\xymatrix@R=0.0000pt@C=30pt{	
		{_{\chi_0}}&{_{\chi_1}}&{_{\chi_l}}&{_{\chi_{\Lambda}}}\\
		{\Circle} \ar@<0.3ex>[r]^{E_1}  &{\Circle}\ar@<0.3ex>[l]^{E_s}  \ar@{.}[r]    &{\Circle} \ar@<0.3ex>[r]^{E_{l+1}} & {\CIRCLE} \ar@<0.3ex>[l]^{E_{l+2}}
}
$$
where the  direction is  $\varepsilon=(1,-1)$,  $l\geq 0$, $\chi_0$ is a leaf of $\sigma(\bB)$ and one of the following holds:
\begin{itemize}
  \item[(i)] $\chi_{0}>0$, $e\mid(\ell_i\cdot p^{n-i}-1)$  and the multiplicity $\mu$ of $X$ is such that $2\leq \mu\leq m$ and 
     $$\mu=
    \begin{cases}
  m+1-\frac{\ell_i\cdot p^{n-i}-1}{e}    & \text{ if $l$ is odd}, \\
   \frac{\ell_i\cdot p^{n-i}-1}{e} +1   &\text{ if $l$ is even};
\end{cases}
   $$
   \item[(ii)]  $\chi_{0}<0$, $e\mid\ell_i$ and the multiplicity  $\mu$ of $X$ is such that $2\leq \mu\leq m$ and 
      $$\mu=
    \begin{cases}
  \frac{\ell_i\cdot p^{n-i}}{e} +1   & \text{ if $l$ is odd}, \\
  m+1- \frac{\ell_i\cdot p^{n-i}}{e}   &\text{ if $l$ is even}.
\end{cases}
   $$
\end{itemize}
  \item[\rm\bf(3)] The module $X$ corresponds to the  path
  $$
\xymatrix@R=0.0000pt@C=30pt{	
		{_{\chi_0}}&{_{\chi_{\Lambda}}}\\
		{\Circle} \ar@<0.3ex>[r]^{E_1}  &{\CIRCLE}\ar@<0.3ex>[l]^{E_2}  
}
$$
where the direction is $\varepsilon=(-1,1)$, $\chi_{\Lambda}$ is a leaf of $\sigma(\bB)$, and one of the following holds:
\begin{itemize}
  \item[(i)] $\chi_{\Lambda}>0$, $e\mid(\ell_i\cdot p^{n-i}-1)$ and the multiplicity $\mu$ of $X$ is such that $2\leq \mu\leq m-1$ and  $\mu=m-\frac{\ell_i\cdot p^{n-i}-1}{e}$;
   \item[(ii)] $\chi_{\Lambda}<0$, $e\mid\ell_i$ and the multiplicity $\mu$ of $X$ is such that $2\leq \mu\leq m-1$ and  \smallskip $\mu=\frac{\ell_i\cdot p^{n-i}}{e}$. 
\end{itemize}
\item[\rm\bf(4)] The module  $X$ corresponds to the  path
$$  \xymatrix@R=0.0000pt@C=30pt{	
     &{_{\chi_0}}&{_{\chi_1}}&{_{\chi_l}}&{_{\chi_{\Lambda}}}\\
      {\Circle}  &{\Circle} \ar@<0.3ex>[r]^{E_1}  \ar@<0.3ex>[l]^{E_{s}}  &{\Circle}\ar@<0.3ex>[l]^{E_{s-1}}  \ar@{.}[r]&{\Circle}\ar@<0.3ex>[r]^{E_{l+1}}&{\CIRCLE}\ar@<0.3ex>[l]^{E_{l+2}}
}
$$
where $l\geq 0$, the successor of $E_1$ around $\chi_0$ is $E_s$, the direction is $\varepsilon=(1,1)$, and one of the following holds:
\begin{itemize}
  \item[(i)] $\chi_{0}>0$, $e\mid(\ell_i\cdot p^{n-i}-1)$ and the multiplicity $\mu$ of $X$ is such that $2\leq \mu\leq m$ and   $$\mu=
    \begin{cases}
  m+1-\frac{\ell_i\cdot p^{n-i}-1}{e}    & \text{ if $l$ is odd}, \\
   \frac{\ell_i\cdot p^{n-i}-1}{e} +1   &\text{ if $l$ is even};
\end{cases}
   $$
   \item[(ii)] $\chi_{0}<0$, $e\mid\ell_i$ and the multiplicity $\mu$ of $X$ is such that $2\leq\mu\leq m$ and  
$$\mu=
    \begin{cases}
  \frac{\ell_i\cdot p^{n-i}}{e} +1   & \text{ if $l$ is odd}, \\
  m+1- \frac{\ell_i\cdot p^{n-i}}{e}   &\text{ if $l$ is even}.
\end{cases}
$$
\end{itemize}
\item[\rm\bf(5)] The module  $X$ corresponds to the  path
$$  \xymatrix@R=0.0000pt@C=30pt{	
      &{_{\chi_0}}&{_{\chi_1}}&{_{\chi_l}}&{_{\chi_{\Lambda}}}\\
      {\Circle} \ar@<0.3ex>[r]^{E_1}  &{\Circle}  \ar@<0.3ex>[r]^{E_{2}}  &{\Circle}\ar@<0.3ex>[l]^{E_{s}}  \ar@{.}[r]&{\Circle}\ar@<0.3ex>[r]^{E_{l+2}}&{\CIRCLE}\ar@<0.3ex>[l]^{E_{l+3}}
}
$$
where $l\geq 0$, the successor of $E_1$ around $\chi_0$ is $E_s$, the direction is $\varepsilon=(-1,-1)$,  and one of the following holds:
\begin{itemize}
  \item[(i)]  $\chi_{0}>0$, $e\mid(\ell_i\cdot p^{n-i}-1)$ and  the multiplicity $\mu$ of $X$ is such that $2\leq\mu\leq m$ and      
  $$\mu=
    \begin{cases}
  m+1-\frac{\ell_i\cdot p^{n-i}-1}{e}    & \text{ if $l$ is odd}, \\
   \frac{\ell_i\cdot p^{n-i}-1}{e} +1   &\text{ if $l$ is even};
\end{cases}
   $$
   \item[(ii)]  $\chi_{0}<0$, $e\mid\ell_i$ and  the multiplicity $\mu$ of $X$ is such that $2\leq\mu\leq m$ and
$$\mu=
    \begin{cases}
  \frac{\ell_i\cdot p^{n-i}}{e} +1   & \text{ if $l$ is odd}, \\
  m+1- \frac{\ell_i\cdot p^{n-i}}{e}   &\text{ if $l$ is even}.
\end{cases}
   $$
\end{itemize}
\item[\rm\bf(6)] The module  $X$ corresponds to the  path
 $$ 
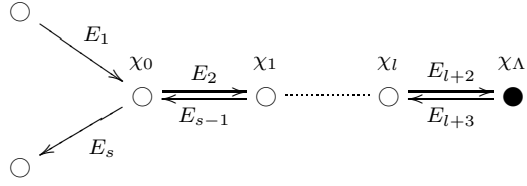
\xymatrix@R=0.0000pt@C=30pt{
 	&& &\\
	{\Circle}\ar[ddr]^{E_{1}} & & &  \\
		&{_{\chi_0}}&{_{\chi_1}}&{_{\chi_l}}&{_{\chi_{\Lambda}}} \\
		&{\Circle}\ar[dddl]^{\:\:E_{s}} \ar@<0.3ex>[r]^{E_2}&{\Circle}\ar@<0.3ex>[l]^{E_{s-1}}\ar@{.}[r]&{\Circle}\ar@<0.3ex>[r]^{E_{l+2}}&{\CIRCLE}\ar@<0.3ex>[l]^{E_{l+3}}\\
		&& &\\
		&& &\\
		{\Circle}& & & \\
		&& &
	}
$$
where $l\geq 0$,  the successor of $E_1$ around $\chi_0$ is $E_s$,  the direction is $\varepsilon=(-1,1)$ and one of the following holds:
\begin{itemize}
  \item[(i)]  $\chi_{0}>0$, $e\mid(\ell_i\cdot p^{n-i}-1)$ and  the multiplicity $\mu$ of $X$ is such that $2\leq\mu\leq m$ and   
  $$\mu=
    \begin{cases}
  m+1-\frac{\ell_i\cdot p^{n-i}-1}{e}    & \text{ if $l$ is odd}, \\
   \frac{\ell_i\cdot p^{n-i}-1}{e} +1   &\text{ if $l$ is even};
\end{cases}
   $$
   \item[(ii)]  $\chi_{0}<0$, $e\mid\ell_i$ and the multiplicity $\mu$ of $X$ is such that $2\leq\mu\leq m$ and
$$\mu=
    \begin{cases}
  \frac{\ell_i\cdot p^{n-i}}{e} +1   & \text{ if $l$ is odd}, \\
  m+1- \frac{\ell_i\cdot p^{n-i}}{e}   &\text{ if $l$ is even}.
\end{cases}
   $$
\end{itemize}
\item[\rm\bf(7)] The module  $X$ corresponds to the  path
$$ 
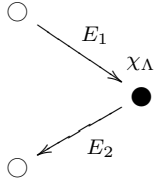
\xymatrix@R=0.0000pt@C=30pt{
                & \\	
		{\Circle}\ar[ddr]^{E_{1}} &  \\
		&{_{\chi_\Lambda}} \\
		&{\CIRCLE}\ar[dddl]^{\:\:E_{2}}\\
		&\\
		&\\
		{\Circle}&  \\
		& 
}
$$
where  the successor of $E_1$ around $\chi_\Lambda$ is $E_2$, the direction is $\varepsilon=(-1,1)$, and one of the following holds:
\begin{itemize}
  \item[(i)]  $\chi_{\Lambda}>0$, $e\mid(\ell_i\cdot p^{n-i}-1)$ and  the multiplicity $\mu$ of $X$ is such that $1\leq\mu\leq m-1$ and $\mu=m-\frac{\ell_i\cdot p^{n-i}-1}{e}$;
  \item[(ii)] $\chi_{\Lambda}<0$, $e\mid\ell_i$ and  the multiplicity $\mu$ of $X$ is such that $1\leq\mu\leq m-1$ and $\mu=\frac{\ell_i\cdot p^{n-i}}{e}$. 
\end{itemize}
\end{itemize}
\end{enumerate}
\end{thm}

\begin{proof}{\ }
\begin{enumerate}
\item[\rm(a)]  If $D\cong C_2$, then the claim was proved in~\S\ref{ssec:char2}. Hence we may assume that $D\ncong C_2$. Since $e=1$ all indecomposable $\bB$-modules are uniserial by Lemma~\ref{lem:complength}(a). Let $X$ be a trivial source $\bB$-module with vertex $D_i$.  Then  $d^+(X)=\ell_i\cdot p^{n-i}-1$ by  Theorem~\ref{thm:mainReform}. If $\chi_1>0$, then $d^+(X)$ is the distance in $\Gamma_s(\bB)$ from $X$ to the hook $S_1$, and if $\chi_1<0$, then $d^-(X)$ is the distance  from $X$ to the hook $\Omega(S_1)$. Thus, the claim follows from \smallskip Lemma~\ref{lem:complength}(c). 
\item[\rm(b)]  Since trivial source modules are liftable, we go through the classification of the liftable $\bB$-modules provided by Theorem~\ref{thm:classific_liftable}(c), where, more precisely, modules of type {\rm(1)} are projective. Hence, indecomposable modules with vertex $D_i$ ($1\leq i\leq n-1$) can only correspond to modules  of type (2), (2'), (3), (4), (5), (6) or \smallskip (7).

\begin{itemize}
  \item[\rm 1.] If $X$ is of type (2), i.e. a hook, then by Corollary~\ref{cor:D1Dn},  $D_i=D$, $W=k$, and $X$  affords a character $\chi\in\Irr^{\circ}(\bB)$ such that $\chi>0$. This \smallskip yields~{\rm\bf(1)}.
  \item[\rm 2.] If $X$ is of type (2'), then $\chi_\Lambda$ is a leaf of $\sigma(\bB)$ and  $X$ is the simple $\bB$-module labelling this leaf. Therefore, by Proposition~\ref{prop:distances}(a) the distance from $X$ to one of the boundaries of $\Gamma_s(\bB)$ is  $(m-1)e$, thus the distance to the other boundary of $\Gamma_s(\bB)$ is 
$$me-1 -\left((m-1)e\right) =e-1\,.$$
Since $e>1$ and $e\mid (p-1)$, we have $1\leq e-1\leq p-2$. However, by Theorem~\ref{thm:mainA}  the distance from a trivial source module  with vertex $D_i$ 
 to each of the boundaries of $\Gamma_s(\bB)$ is either~$0$ or greater or equal to $p-1$. It follows that $X$ is never a trivial source \smallskip module.
  \item[\rm 3.]  If $X$ is of type (3) with $l\geq 0$ and $\chi_0$ is a leaf of $\sigma(\bB)$, then by Proposition~\ref{prop:distances}(b),  the distance from $X$ to the hook $E_1$ is given by 
$$d(X,E_1)=\begin{cases}
  e(m-\mu+1)    & \text{if $l$ is odd}, \\
   e(\mu-1)   & \text{if $l$ is even}.
\end{cases}$$
Now, $E_1$ affords the character $\chi_0$. Therefore, if $\chi_0>0$, then $d(X,E_1)=d^{+}(X)$. Hence,  by  Theorem~\ref{thm:mainA}, $X$ is a trivial source module if  $d(X,E_1)=\ell_i\cdot p^{n-i}-1$, that is if $e\mid(\ell_i\cdot p^{n-i}-1)$ and 
      $$\mu=
    \begin{cases}
  m+1-\frac{\ell_i\cdot p^{n-i}-1}{e}    & \text{ if $l$ is odd}, \\
   \frac{\ell_i\cdot p^{n-i}-1}{e} +1   &\text{ if $l$ is even}.
\end{cases}
   $$
If $\chi_0<0$, then $d(X,E_1)=d^{-}(X)$. Hence,   by  Theorem~\ref{thm:mainA}, $X$ is a trivial source module if  $\ell_i\cdot p^{n-i}-1=d^+(X)=em-1-d(X,E_1)$, that is if $e\mid\ell_i$ and 
      $$\mu=
    \begin{cases}
  \frac{\ell_i\cdot p^{n-i}}{e} +1   & \text{ if $l$ is odd}, \\
  m+1- \frac{\ell_i\cdot p^{n-i}}{e}   &\text{ if $l$ is even}.
\end{cases}
   $$
   This \smallskip yields~{\rm\bf(2)}. 
  \item[\rm 4.]  If $X$ is of type (3) with $l=0$ and $\chi_{\Lambda}$ is a leaf of $\sigma(\bB)$, then by Proposition~\ref{prop:distances}(b')  the distance from $X$ to the unique hook $H$ which is uniserial of length $m$ with all composition factors isomorphic to  $E_1=E_2$ is
 $$d(X,H)=e(m-\mu)\,.$$
Now, $H$ affords the character $\chi_\Lambda$. Therefore, if $\chi_\Lambda>0$, then $d(X,H)=d^{+}(X)$. Hence, by  Theorem~\ref{thm:mainA}, $X$ is a trivial source module if  
$$\ell_i\cdot p^{n-i}-1=e(m-\mu),\,$$ i.e. if  $e\mid(\ell_i\cdot p^{n-i}-1)$ and $\mu=m-\frac{\ell_i\cdot p^{n-i}-1}{e}$.\\
 If $\chi_\Lambda<0$, then $d(X,H)=d^{-}(X)$. Hence, by  Theorem~\ref{thm:mainA}, $X$ is a trivial source module if  $\ell_i\cdot p^{n-i}-1=d^+(X)=em-1-d(X,E_1)$, that is if $e\mid\ell_i$ and 
  $\mu=\frac{\ell_i\cdot p^{n-i}}{e}$.
     This \smallskip yields~{\rm\bf(3)}. 
  \item[\rm 5.] If $X$ is of type (4), then by Proposition~\ref{prop:distances}(c)  the distance from $X$ to the  unique hook~$H$ with  socle $E_1$ and head~$E_s$ is given by
$$
d(X,H)=
\left\{
	\begin{array}{lr}
		e(m-\mu+1)&\text{if } l\mbox{ is odd,}\\
		e(\mu-1)&\text{if } l \mbox{ is even}.
	\end{array}\right.
$$
The hook $H$ affords the character $\chi_{0}$. Therefore, the same computations as in  Case~3 above  \smallskip yield~{\rm\bf(4)}.

\item[\rm6.]  If $X$ is of type (5), then by Proposition~\ref{prop:distances}(d)  the distance from $X$ to the  unique hook~$H$ with  socle $E_1$ and head~$E_s$  is given by
$$
d(X,H)=
\left\{
	\begin{array}{lr}
		e(m-\mu+1)&\text{if } l\mbox{ is odd,}\\
		e(\mu-1)&\text{if } l \mbox{ is even}.
	\end{array}\right.
$$
The hook $H$ affords the character $\chi_{0}$. Therefore, the same computations as in  Case~3 above  \smallskip yield~{\rm\bf(5)}.
  \item[\rm 7.]   If $X$ is of type (6), then by Proposition~\ref{prop:distances}(e)  the distance from $X$ to the  unique hook~$H$ with socle $E_1$  and head $E_s$ is given by
$$
d(X,H)=
\left\{
	\begin{array}{lr}
		e(m-\mu+1)&\text{if } l\mbox{ is odd,}\\
		e(\mu-1)&\text{if } l \mbox{ is even}.
	\end{array}\right.
$$
The hook $H$ affords the character $\chi_{0}$. Therefore, the same computations as in  Case~3 above  \smallskip yield~{\rm\bf(6)}.

  \item[\rm 8.]  If $X$ is of type (7), then by Proposition~\ref{prop:distances}(f), the distance from $X$ to the unique hook  $H$ with socle $E_1$  and head $E_2$ is
$$d(X,H)=e(m-\mu)\,.$$
Moreover, the hook $H$ affords the character $\chi_\Lambda$.  Therefore, the same computations as in  Case~4 above  yield~{\rm\bf(7)}.
\end{itemize}
\end{enumerate}
\end{proof}

\begin{rem}
Theorem~\ref{thm:classific_liftable}(b) yields $\ell_i\cdot p^{n-i}-1\equiv -1 \pmod{e}$ or $\ell_i\cdot p^{n-i}-1\equiv 0 \pmod{e}$. In the former case $e\mid \ell_i$ as $e\mid (p-1)$ and in the latter case $e\mid (\ell_i\cdot p^{n-i}-1)$. 
\end{rem}

\begin{rem}[Cotrivial source modules]
If $X$ is a non-projective indecomposable cotrivial source $\bB$-module with vertex $D_i$,  then $\Omega(X)$ is a  non-projective indecomposable trivial source $\bB$-module with vertex $D_i$. Thus, $d^+(X)=p^n-1-d^+(\Omega(X))=p^{n}-\ell_i\cdot p^{n-i}$. It follows that cotrivial source modules can be classified in a similar fashion, replacing $\ell_i\cdot p^{n-i}-1$ with $p^{n}-\ell_i\cdot p^{n-i}$ in the proof of Theorem~\ref{thm:mainB}. 
\end{rem}

\newpage

\appendix 

\section{The classification of the indecomposable liftable modules in blocks with cyclic defect groups} \label{AppA}

We recall here a result of the first author and Naehrig \cite{HN12} classifying the indecomposable liftable modules in blocks with cyclic defect groups.
A minor correction must be brought to the original statement in case the exceptional vertex is a leaf of $\sigma(\bB)$.\\

The notation in use below to parametrise the indecomposable $\bB$-modules is based on standard results of Janusz \cite[\S5]{Jan69} and more recent work of Bleher-Chinburg~\cite{BleChi}. 
The main idea is as follows: following Janusz \cite[\S5]{Jan69}, each indecomposable $\bB$-module $X$ which is neither projective nor simple can be encoded using a \emph{path} on $\sigma(\bB)$, which is by definition  a certain connected subgraph of $\sigma(\bB)$. This path may be seen as an ordered sequence $(E_1,\ldots,E_s)$ of edges of  $\sigma(\bB)$, called \emph{top-socle sequence} of $X$, and  where $E_i,E_{i+1}$ have a common vertex for every $1\leq i\leq s-1$, the odd-labelled edges are in the head of $X$ and the even-labelled edge is in the socle of $X$, or conversely, and some edges may be passed twice if necessary.
Moroever \cite{BleChi} associates to each indecomposable $\bB$-module $X$ two further parameters:  a \emph{direction} $\varepsilon=(\varepsilon_1,\varepsilon_s)$ and a \emph{multiplicity} $\mu$. For  $i\in\{1,s\}$ we set $\varepsilon_i=1$ if $E_i$ is in the head of $X$ and $\varepsilon_i=-1$ if $E_i$ is in the socle of $X$. If $m=1$, then $\mu:=0$. If $m>1$, then $\mu$ corresponds to the number of times that a simple module $E_{j}$ connected to the exceptional vertex  occurs as a composition factor of~$X$ (this is independent of the choice of $E_j$).  The module $X$ is entirely parametrised by its top-socle sequence (i.e. path), direction and multiplicity.\\

\begin{thm}[{}{\cite[Theorem 2.1]{HN12} with correction to (c)(3)}]\label{thm:classific_liftable}
Let $\bB$ be a cyclic block with defect $n\geq 1$, $e$ simple modules, and exceptional multiplicity $m:=(p^n-1)/e$.
\begin{enumerate}[ \rm(a)]
\item
The number of indecomposable liftable $\bB$-modules equals $m+1$ if $e=1$, and $e(2m+1)$ if $e>1$. 
If $e=1$ all indecomposable $\bB$-modules are liftable. We thus assume that $e>1$ in the following.
\item 
Let $X$ be an indecomposable $\bB$-module. Then $X$ is liftable if and only if the minimal distance of $X$ to the boundary of $\Gamma_s(\bB)$ is of the form $ei$ for some $0\leq i\leq \lfloor(m-1)/2\rfloor$ or of the form 
$ei-1$ for some $0\leq i\leq \lfloor m/2\rfloor$.
\item 
Let $X$ be an indecomposable $\bB$-module. Then $X$ is liftable, if and only if it belongs to one of the types described in {\rm(1)}--{\rm(7)} below. In types {\rm(2'), \rm(3)}--{\rm(7)}, $m>1$ holds. 

\begin{enumerate}
  \item[\rm(1)] The module $X$ is projective.
  \item[\rm(2)] The module $X$ is a hook; in particular,  $X$ is uniserial with descending composition series corresponding to a counter-clockwise walk around a vertex $\chi\in\{\chi_1,\ldots,\chi_e,\chi_{\Lambda}\}$ of $\sigma(\bB)$, where each composition factor occurs with multiplicity $m$ if $\chi=\chi_{\Lambda}$. The character of any lift of $X$ is $\chi$.\\
  The number of modules of this type is $2e$. These are exactly the modules lying at the boundary of $\Gamma_s(\bB)$.
\item[\rm(2')] 
 In case $\chi_{\Lambda}$ is a leaf of $\sigma(\bB)$, then the module $X$ is the simple module labelling this leaf. In this case the character of any lift of $X$ is an exceptional character.
  \item[\rm(3)]  The module $X$ corresponds to the  path 
  $$
\xymatrix@R=0.0000pt@C=30pt{	
		{_{\chi_0}}&{_{\chi_1}}&{_{\chi_l}}&{_{\chi_{\Lambda}}}\\
		{\Circle} \ar@<0.3ex>[r]^{E_1}  &{\Circle}\ar@<0.3ex>[l]^{E_s}  \ar@{.}[r]    &{\Circle} \ar@<0.3ex>[r]^{E_{l+1}} & {\CIRCLE} \ar@<0.3ex>[l]^{E_{l+2}}
}
$$
where $l\geq 0$. In addition, in case $l>0$, then $\chi_0$ is a leaf of $\sigma(\bB)$, in case $l=0$ either $\chi_0$ or $\chi_{\Lambda}$ is a leaf. Moreover, 
the direction is $\varepsilon=(1,-1)$ and  the multiplicity $\mu$ staisfies:
\begin{enumerate}[\rm(i)]
\item  $2\leq \mu\leq m$ if $l\geq 0$ and $\chi_{0}$ is a leaf; and 
\item  $2\leq \mu\leq m-1$ if $l=0$ and $\chi_{\Lambda}$ is a leaf. 
\end{enumerate} 
  \item[\rm(4)] The module $X$ corresponds to the  path
$$  
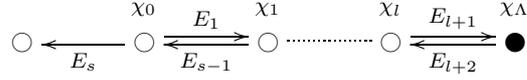
\xymatrix@R=0.0000pt@C=30pt{	
      &{_{\chi_0}}&{_{\chi_1}}&{_{\chi_l}}&{_{\chi_{\Lambda}}}\\
      {\Circle}  &{\Circle} \ar@<0.3ex>[r]^{E_1}  \ar@<0.3ex>[l]^{E_{s}}  &{\Circle}\ar@<0.3ex>[l]^{E_{s-1}}  \ar@{.}[r]&{\Circle}\ar@<0.3ex>[r]^{E_{l+1}}&{\CIRCLE}\ar@<0.3ex>[l]^{E_{l+2}}
}
$$
where $l\geq 0$, the successor of $E_1$ around $\chi_0$ is $E_s$, the direction is $\varepsilon=(1,1)$, and $2\leq \mu\leq m$.
  \item[\rm(5)] The module     
 $X$ corresponds to the  path 
$$  
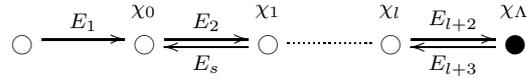
\xymatrix@R=0.0000pt@C=30pt{	
      &{_{\chi_0}}&{_{\chi_1}}&{_{\chi_l}}&{_{\chi_{\Lambda}}}\\
      {\Circle} \ar@<0.3ex>[r]^{E_1}  &{\Circle}  \ar@<0.3ex>[r]^{E_{2}}  &{\Circle}\ar@<0.3ex>[l]^{E_{s}}  \ar@{.}[r]&{\Circle}\ar@<0.3ex>[r]^{E_{l+2}}&{\CIRCLE}\ar@<0.3ex>[l]^{E_{l+3}}
}
$$
where $l\geq 0$, the successor of $E_1$ around $\chi_0$ is $E_s$, the direction is $\varepsilon=(-1,-1)$, and $2\leq \mu\leq m$.
 \item[\rm(6)]   The module 
 $X$ corresponds to the  path
 $$ 
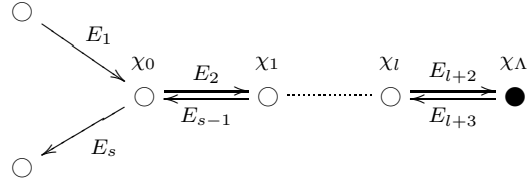
\xymatrix@R=0.0000pt@C=30pt{	
		{\Circle}\ar[ddr]^{E_{1}} & & &  \\
		&{_{\chi_0}}&{_{\chi_1}}&{_{\chi_l}}&{_{\chi_{\Lambda}}} \\
		&{\Circle}\ar[dddl]^{\:\:E_{s}} \ar@<0.3ex>[r]^{E_2}&{\Circle}\ar@<0.3ex>[l]^{E_{s-1}}\ar@{.}[r]&{\Circle}\ar@<0.3ex>[r]^{E_{l+2}}&{\CIRCLE}\ar@<0.3ex>[l]^{E_{l+3}}\\
		&& &\\
		&& &\\
		{\Circle}& & & 
}
$$
where $l\geq 0$,  the successor of $E_1$ around $\chi_0$ is $E_s$, the direction is $\varepsilon=(-1,1)$, and the multiplicity is $2\leq \mu\leq m$.
 \item[\rm(7)]The module  $X$ corresponds to the  path
 $$ 
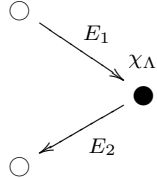
\xymatrix@R=0.0000pt@C=30pt{
		{\Circle}\ar[ddr]^{E_{1}} &  \\
		&{_{\chi_{\Lambda}}} \\
		&{\CIRCLE}\ar[dddl]^{\:\:E_{2}}\\
		&\\
		&\\
		{\Circle} &   
}
$$
where  the successor of $E_1$ around $\chi_{\Lambda}$ is $E_2$, the direction is $\varepsilon=(-1,1)$, and  the multiplicity is \smallskip $1\leq \mu\leq m-1$. 
\end{enumerate}
Each path in {\rm(3), \rm(4), \rm(5), \rm(6) or \rm(7)} gives rise to $m-1$ modules, distinguished by their multiplicity~$\mu$, unless the path is of type {\rm(3)(ii)}, in which case there are $m-2$ \smallskip modules.
\item Let $X$ be an indecomposable liftable $\bB$-module as in {\rm(c)},{\rm(3)}--{\rm(7)}.  Then the character of any lift of $X$ is of the form 
$$\sum_{i=0}^{l}\chi_i+\Xi\,,$$
 where  $\Xi$ is a sum of $\mu-1$ distinct exceptional characters if $X$ is as in {\rm(3)(i)}, {\rm(4), (5)}, or {\rm(6)}, whereas 
 the character of any lift of $X$ is of the form $\Xi\,,$ 
 where $\Xi$ is a sum of~$\mu$ distinct exceptional characters if $X$ is as in {\rm(3)(ii)}, or {\rm(7)}. 
\end{enumerate}
\end{thm}

\begin{proof}
Only Part {\rm(c)(3)} and {\rm(d)} are modified with respect to the original statement.  Case {\rm(c)(2')} is added in order  to consider the case that the exceptional vertex is a leaf of $\sigma(\bB)$.  More precisely:
\begin{itemize}
\item[$\cdot$] if the exceptional vertex of $\sigma(\bB)$ is not a leaf, then there is nothing to change;
\item[$\cdot$] if the exceptional vertex of $\sigma(\bB)$ is a leaf, then the simple module $E$ labelling this leaf is not a hook, so that the r\^{o}les of $E$ and the hook with $\mu=m$ composition factors equal to $E$  need to be swapped in the original proof.
\end{itemize}
The form of the ordinary characters in {\rm(d)} is then an immediate consequence of the modification brought to {\rm(c)(3)}.
\end{proof}

\noindent Finally, we note that in Theorem~\ref{thm:classific_liftable} we have also slightly altered the original statement of \cite[Theorem 2.1]{HN12} by splitting the original case {\rm(c)(6)} into cases {\rm(c)(6)} and {\rm(c)(7)} in order to facilitate the computations in Appendix~\ref{Ap:distances} below.


\vspace{6mm}
\section{Distances  in the stable Auslander-Reiten quiver}\label{Ap:distances}\label{AppB}

Below 
 $d(X,H)$ denotes the distance in $\Gamma_s(\bB)$ from the non-projective indecomposable $\bB$-module $X$  to the boundary  to which  the  hook $H$ belongs (i.e. the length of a shortest path between both modules). We now go through the list of liftable $\bB$-modules provided by Theorem~\ref{thm:classific_liftable} and give their distances to one of the boundaries of  $\Gamma_s(\bB)$.
 \par
 If $e=1$, then by Theorem~\ref{thm:classific_liftable}(a) all non-projective indecomposable $\bB$-modules are liftable, and their distance to the boundary of of  $\Gamma_s(\bB)$ containing the unique simple $\bB$-module is given by their composition length minus 1 by Lemma~\ref{lem:complength}(c). Thus  we may assume that $e>1$. If $m=1$, then by Theorem~\ref{thm:classific_liftable}, the non-projective indecomposable  liftable $\bB$-modules are precisely the hooks, hence we may also assume that $m>1$.

\begin{prop}\label{prop:distances}
Let $\mathbf B$ be a cyclic block with defect group $D\cong C_{p^n}$, $e>1$ simple modules and exceptional multiplicity $m>1$. Let $X$ be a non-projective indecomposable liftable $\mathbf B$-module, which is  not a hook. 
\begin{enumerate}
\item[\rm(a)] 
Assume $\chi_{\Lambda}$ is a leaf of $\sigma(\bB)$, $X$ is the simple module labelling this leaf, and $H$ is the hook corresponding to $\chi_{\Lambda}$ (i.e. the uniserial module of length $m$ with composition factors all isomorphic to $X$), then 
$$d(X,H)=e(m-1)=|D|-1-e\,.$$
\item[\rm(b)]
Assume  $X$ corresponds to the  path
  $$
\xymatrix@R=0.0000pt@C=30pt{	
		{_{\chi_0}}&{_{\chi_1}}&{_{\chi_l}}&{_{\chi_{\Lambda}}}\\
		{\Circle} \ar@<0.3ex>[r]^{E_1}  &{\Circle}\ar@<0.3ex>[l]^{E_s}  \ar@{.}[r]    &{\Circle} \ar@<0.3ex>[r]^{E_{l+1}} & {\CIRCLE} \ar@<0.3ex>[l]^{E_{l+2}}
}
$$
where  the direction is $\varepsilon=(1,-1)$, $l\geq 0$, and $\chi_0$ is a leaf of $\sigma(\bB)$, so that $2\leq \mu\leq m$. Then the distance from $X$ to the hook $H:=E_1$ (simple) is 
$$d(X,H)=\begin{cases}
  e(m-\mu+1)    & \text{if $l$ is odd}, \\
   e(\mu-1)   & \text{if $l$ is even}.
\end{cases}$$
\item[\rm(b')]
Assume  $X$ corresponds to the  path
  $$
\xymatrix@R=0.0000pt@C=30pt{	
		{_{\chi_0}}&{_{\chi_{\Lambda}}}\\
		{\Circle} \ar@<0.3ex>[r]^{E_1}  &{\CIRCLE}\ar@<0.3ex>[l]^{E_2}  
}
$$
where $\chi_{\Lambda}$ is a leaf of $\sigma(\bB)$, the direction is $\varepsilon=(-1,1)$  and the multiplicity is $2\leq \mu\leq m-1$. Then the distance from $X$ to the unique hook $H$ which is uniserial of length $m$ and has composition factors all isomorphic to  $E_1=E_2$ is
 $$d(X,H)=e(m-\mu)\,.$$
\item[\rm(c)]
Assume  $X$ corresponds to the  path
$$  \xymatrix@R=0.0000pt@C=30pt{	
    {_{\chi_{u0}}}  &{_{\chi_0}}&{_{\chi_1}}&{_{\chi_l}}&{_{\chi_{\Lambda}}}\\
      {\Circle}  &{\Circle} \ar@<0.3ex>[r]^{E_1}  \ar@<0.3ex>[l]^{E_{s}}  &{\Circle}\ar@<0.3ex>[l]^{E_{s-1}}  \ar@{.}[r]&{\Circle}\ar@<0.3ex>[r]^{E_{l+1}}&{\CIRCLE}\ar@<0.3ex>[l]^{E_{l+2}}
}
$$
where $l\geq 0$, the successor of $E_1$ around $\chi_0$ is $E_s$, the direction is $\varepsilon=(1,1)$, and $2\leq \mu\leq m$.\\
Then the distance from $X$ to the unique hook  $H$ with  socle $E_1$ and head $E_s$  is
$$
d(X,H)=
\left\{
	\begin{array}{lr}
		e(m-\mu+1)&\text{if } l\mbox{ is odd,}\\
		e(\mu-1)&\text{if } l \mbox{ is even.}
	\end{array}\right.
$$
\item[\rm(d)]
 Assume  $X$ corresponds to the  path   
$$  \xymatrix@R=0.0000pt@C=30pt{	
    {_{\chi_{u0}}}  &{_{\chi_0}}&{_{\chi_1}}&{_{\chi_l}}&{_{\chi_{\Lambda}}}\\
      {\Circle} \ar@<0.3ex>[r]^{E_1}  &{\Circle}  \ar@<0.3ex>[r]^{E_{2}}  &{\Circle}\ar@<0.3ex>[l]^{E_{s}}  \ar@{.}[r]&{\Circle}\ar@<0.3ex>[r]^{E_{l+2}}&{\CIRCLE}\ar@<0.3ex>[l]^{E_{l+3}}
}
$$
where $l\geq 0$, the successor of $E_1$ around $\chi_0$ is $E_s$, the direction is $\varepsilon=(-1,-1)$, and $2\leq \mu\leq m$. Then the distance from $X$ to the unique hook  $H$ with  socle $E_1$ and head $E_s$ is 
$$
d(X,H)=
\left\{
	\begin{array}{lr}
		e(m-\mu+1)&\text{if } l\mbox{ is odd,}\\
		e(\mu-1)&\text{if } l \mbox{ is even.}
	\end{array}\right.
$$
\item[\rm(e)] 
Assume  $X$ corresponds to the  path
 $$ \xymatrix@R=0.0000pt@C=30pt{
 	{_{\chi_{u0}}}&& &\\
	{\Circle}\ar[ddr]^{E_{1}} & & &  \\
		&{_{\chi_0}}&{_{\chi_1}}&{_{\chi_l}}&{_{\chi_{\Lambda}}} \\
		&{\Circle}\ar[dddl]^{\:\:E_{s}} \ar@<0.3ex>[r]^{E_2}&{\Circle}\ar@<0.3ex>[l]^{E_{s-1}}\ar@{.}[r]&{\Circle}\ar@<0.3ex>[r]^{E_{l+2}}&{\CIRCLE}\ar@<0.3ex>[l]^{E_{l+3}}\\
		&& &\\
		&& &\\
		{\Circle}& & & \\
		{^{\chi_{d0}}}&& &
	}
$$
where $l\geq 0$,  the successor of $E_1$ around $\chi_0$ is $E_s$, the direction is $\varepsilon=(-1,1)$,  $2\leq \mu\leq m$.   Then the distance from $X$ to the unique hook  $H$ with socle $E_1$  and head $E_s$ is
$$
d(X,H)=
\left\{
	\begin{array}{lr}
		e(m-\mu+1)&\text{if } l\mbox{ is odd,}\\
		e(\mu-1)&\text{if } l \mbox{ is even.}
	\end{array}\right.
$$
\item[\rm(f)] 
Assume  $X$ corresponds to the  path
 $$ 
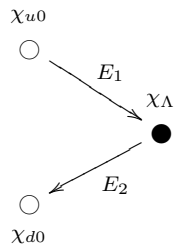
\xymatrix@R=0.0000pt@C=30pt{
                {_{\chi_{u0}}}& \\	
		{\Circle}\ar[ddr]^{E_{1}} &  \\
		&{_{\chi_{\Lambda}}} \\
		&{\CIRCLE}\ar[dddl]^{\:\:E_{2}}\\
		&\\
		&\\
		{\Circle}&  \\
		{^{\chi_{d0}}}& 
}
$$
where  the successor of $E_1$ around $\chi_{\Lambda}$ is $E_2$, the direction is $\varepsilon=(-1,1)$, and  $1\leq \mu\leq m-1$. 
Then the distance from $X$ to the unique hook  $H$ with socle $E_1$  and head $E_s$ is
$$d(X,H)=e(m-\mu)\,.$$
\end{enumerate}
In all cases, the distances to the other boundary of $\Gamma_s(\bB)$ is $me-1-d(X,H)$.
\end{prop}

\noindent Notice that in (b') we changed the direction with respect to the notation of Theorem~\ref{thm:classific_liftable}, but this is not relevant because (b') describes uniserial modules all of whose composition factors are isomorphic, and in particular have the same head and socle.

\begin{proof}
If $X$ is a hook, then $X$ lies on the boundary of $\Gamma_s(\bB)$ by \S\ref{ssec:PIMs}, hence (a).  In all other cases in order to determine $d(X,H)$ we apply \cite[Theorem 3.5]{BleChi}.   More precisely, following the notation of \cite[Definion~3.2 and Definition 3.4]{BleChi}, in each case we need to determine the labelling $v_a,v_z,S_a,S_z$, the walk $W_X$ and its length, and the \smallskip parameter~$\eta$.
	
\begin{enumerate}
\item[\rm(a)]  Because $X$ is simple, the walk $W_X$ of \cite[Definition 3.4]{BleChi} is of length $n=1$ and the parameter $\eta$ of \cite[Definition 3.4]{BleChi} is given by $\eta=m-1$. It follows from  \cite[Theorem 3.5]{BleChi} that 
$$d(X,H)=(n-1)/2+\eta e=e(m-1)\,.$$ 
\item[\rm(b)] In this case $v_a=\chi_1$, $v_z=\chi_0$, and $S_a=E_1=S_z$. It follows that the walk
	$W_X$ is of length $n=2e+1$. Moreover, the parameter $\eta$ is given by 
	\[\eta=\left\{\begin{array}{lr}
		m-\mu&l\mbox{ is odd}\\
		\mu-2&l\mbox{ is even}.
		\end{array}\right.
	\]
	Hence \cite[Theorem 3.5]{BleChi} yields
	$$d(X,H)=e(1+\eta)=\begin{cases}
  e(m-\mu+1)    & \text{if $l$ is odd}, \\
   e(\mu-1)   & \text{if $l$ is even}.
\end{cases}$$
\item[\rm(b')] In this case $v_a=\chi_{0}$, $v_z=\chi_{\Lambda}$, $S_a=E_2$, $S_z=E_1$. Thus, the walk
	$W_X$  is of length $n=1$. Moreover, $\eta=m-\mu$.  Hence \cite[Theorem 3.5]{BleChi} yields
	$$d(X,H)=(n-1)/2+\eta e=e(m-\mu)\,.$$
\item[\rm(c)] In this case $v_a=\chi_{u0}$, $v_z=\chi_{0}$, and $S_a=E_s=S_z$. Thus, the walk $W_X$  is of length $n=2e+1$, and the parameter $\eta$ is given by 
	\[\eta=\left\{\begin{array}{lr}
		m-\mu&l\mbox{ is odd}\\
		\mu-2&l\mbox{ is even}.
		\end{array}\right.
	\]
	Therefore, by \cite[Theorem 3.5]{BleChi}, 
	$$d(X,H)=(n-1)/2+\eta e= e+ \eta e =
	\begin{cases}
e(m-\mu+1)&\text{if } l\mbox{ is odd,}\\
		e(\mu-1)-1&\text{if } l \mbox{ is even.}
\end{cases}$$
\item[\rm(d)]   In this case $v_a=\chi_{1}$, $v_z=\chi_{0}$, and $S_a=E_s=S_z$. Thus, the walk $W_X$ is of length $n=2e+1$, and the parameter $\eta$ is given by 
	\[\eta=\left\{\begin{array}{lr}
		m-\mu&\text{if }l\mbox{ is odd}\\
		\mu-2&\text{if }l\mbox{ is even}.
		\end{array}\right.
	\]
Therefore, by \cite[Theorem 3.5]{BleChi}, 
	$$d(X,H)=(n-1)/2+\eta e= e+ \eta e =
	\begin{cases}
e(m-\mu+1)&\text{if } l\mbox{ is odd,}\\
		e(\mu-1)-1&\text{if } l \mbox{ is even.}
\end{cases}$$
 \item[\rm(e)] In this case, we see that $v_a=\chi_{d0}$, $v_z=\chi_{0}$, and $S_a=E_s=S_z$. Thus, the walk $W_X$  is of length $n=2e+1$, and the parameter $\eta$ is given by 
	\[\eta=\left\{\begin{array}{lr}
		m-\mu&\text{if } l\mbox{ is odd}\\
		\mu-2&\text{if } l\mbox{ is even}.
		\end{array}\right.
	\]
Therefore, by \cite[Theorem 3.5]{BleChi}, 
	$$d(X,H)=(n-1)/2+\eta e= e+ \eta e =
	\begin{cases}
e(m-\mu+1)&\text{if } l\mbox{ is odd,}\\
		e(\mu-1)-1&\text{if } l \mbox{ is even.}
\end{cases}$$
\item[\rm(f)]  In this case $v_a=\chi_{d0}$, $v_z=\chi_{\Lambda}$, and $S_a=E_2=S_z$. Thus, the walk $W_X$  is of length $n=1$, and  $\eta=m-\mu$.  Therefore, \cite[Theorem 3.5]{BleChi} yields 
	$$d(X,H)=e\eta = e(m-\mu)\,.$$
\end{enumerate} 
\end{proof}
\bigskip

\vspace{1cm}

\noindent\textbf{Acknowledgments.} The authors are grateful to Natalie Naehrig  for allowing them to include in Appendix B an expanded version of her preliminary computations on the distances of liftable modules to the rim of the stable Auslander-Reiten quiver, and to Shigeo Koshitani for numerous examples/discussions on the stable Auslauder-Reiten quiver of cyclic blocks since February 2014. 
The authors also wish to thank Markus Linckelmann for his detailed explanations on the source algebras of cyclic blocks, and Klaus Lux for pointing out useful references.   The second author is grateful to the Lehrstuhl~D f\"ur Mathematik of the RWTH Aachen for its hospitality during the preparation of the present article. Finally the authors would like to thank the referee for several helpful comments that have improved the exposition of the paper.

\bigskip
\bigskip






\end{document}